\documentclass[reqno, 12pt]{amsart}
\usepackage{amsmath,amssymb}
\usepackage{graphicx}
\usepackage{verbatim}
\usepackage[colorlinks,linkcolor=blue,anchorcolor=red,citecolor=red]{hyperref}
\usepackage{enumerate}
\usepackage{upgreek}

\newtheorem{df}{Definition}[section]
\newtheorem{lm}{Lemma}[section]
\newtheorem{remark}{Remark}[section]
\newtheorem{thm}{Theorem}[section]
\newtheorem{rmk}{Remark}[section]
\newtheorem{prop}{{\bf Proposition}}[section]
\newtheorem{cor}{Corollary}[section]

\DeclareMathOperator{\Id}{Id}
\DeclareMathOperator{\Ima}{Im}
\newcounter{saveeqn}%

\makeatletter
\evensidemargin
\oddsidemargin
\makeatother

\title [Takens Theorem  for nonautonomous dynamical systems]{Takens Theorem  for nonautonomous partially hyperbolic
dynamical systems}

\author{ Davor Dragi\v cevi\'c}
\address[Davor Dragi\v cevi\'c]
{Faculty of Mathematics\\
    University of Rijeka\\
    Rijeka 51000, Croatia}
\email[D.~Dragi\v cevi\'c]{ddragicevic@math.uniri.hr}

\author{Xiao Tang}
\address[Xiao Tang]
{School of Mathematical Sciences\\
    Chongqing Normal University\\
    Chongqing 401331, China}
\email[X. Tang (corresponding author)]{mathtx@163.com}

\author{ Wenmeng Zhang}
\address[Wenmeng Zhang]
{School of Mathematical Sciences\\
    Chongqing Normal University\\
    Chongqing 401331, China}
\email[W.~M.~Zhang]{wmzhang@cqnu.edu.cn}

\date{}

\begin{document}
\maketitle

\begin{abstract}
Takens Theorem for a partially hyperbolic dynamics provides a  normal linearization along the center manifold. In this paper, we give the nonautonomous version of Takens Theorem under non-resonance conditions formulated in terms of the dichotomy spectrum. In our proof, one difficulty is to solve homological equations for the normal form theory which involve a center variable, while another difficulty is to find the dichotomy spectrum of   a certain matrix cocycle that is block lower triangular.
In order to overcome those difficulties, in comparison with the autonomous case, we need an additional term in the (nonautonomous) non-resonance conditions to guarantee certain spectral gap conditions. This additional term
disappears naturally in the autonomous case.

\vskip 0.2cm


\vskip 0.2cm
{\bf AMS (2010) subject classification:} 37D10
\end{abstract}

\baselineskip 15pt   
\parskip 10pt         

\thispagestyle{empty}
\setcounter{page}{1}


\setcounter{equation}{0}
\setcounter{lm}{0}
\setcounter{thm}{0}
\setcounter{rmk}{0}
\setcounter{df}{0}
\setcounter{cor}{0}
\setcounter{exa}{0}
\allowdisplaybreaks[2]


\section{Introduction}

\noindent Smooth linearization is one of the most important topics in the qualitative theory of dynamical systems since it is highly related to the classification, simplification and the structural stability of dynamical systems. Its study goes back to the pioneering work of Poincar\'e (\cite{Poin}), who proved that an analytic diffeomorphism defined on $\mathbb{C}^d$ ($d\in\mathbb{N}$) can be analytically conjugated to
its linear part provided that  its eigenvalues lie in the Poincar\'e domain, i.e., all of the eigenvalues $\lambda_i$ for $i=1,\ldots ,d$ have modules less than $1$ or greater than $1$ (which is a special case of the hyperbolicity) and satisfy the
{\it non-resonance conditions of all orders}:
\begin{align}
\lambda_1^{t_1}\cdots \lambda_d^{t_d}\ne \lambda_j\qquad \forall j=1,\ldots ,d,
\label{NR-cond}
\end{align}
where $t_i\ge 0$ are all possible integers such that
$\sum_{i=1}^dt_i\ge 2$.
On $\mathbb{R}^d$,
Sternberg (\cite{Sternberg1,Sternberg2})
proved that for any $k\in \mathbb{N}$ there is a sufficiently large  $N\in \mathbb N$ (greater than $k$)
such that each $C^N$ hyperbolic diffeomorphism is locally $C^k$ conjugated to its linear part if the
non-resonance condition of order $N$ holds, i.e., (\ref{NR-cond}) holds for integers $t_i\ge 0$ such that $2\le \sum_{i=1}^dt_i\le N$. Remark that $k\to \infty$ as $N\to \infty$.

In the non-hyperbolic case, Takens (\cite{Takens-Topol71}) further considered diffeomorphisms with center directions. More precisely, he proved  that for any $k\in \mathbb{N}$ there is a sufficiently large  $N\in \mathbb N$
such that if
\begin{align}
|\lambda_1^{t_1}\cdots \lambda_d^{t_d}|\ne |\lambda_j|,\qquad \forall j=1,\ldots, d,
\label{S-NR}
\end{align}
for integers $t_i\ge 0$ such that $2\le \sum_{i=1}^dt_i\le N$,
then a $C^N$ diffeomorphism with a center direction is locally $C^k$ conjugated to its {\it Takens normal form}
\begin{align*}
\left(
  \begin{array}{lll}
x_s
\vspace{1ex}\\
x_c
\vspace{1ex}\\
x_u
\end{array}
\right)
\mapsto
\left(
  \begin{array}{lll}
\Lambda_s(x_c)x_s
\vspace{1ex}\\
\Lambda_cx_c+f_c(x_c)
\vspace{1ex}\\
\Lambda_u(x_c)x_u
\end{array}
\right),
\end{align*}
where $x_*$ ($*:=s,c,u$) are variables corresponding to the stable, center and unstable subspaces
and $\Lambda_*$ ($*:=s,c,u$) are linear operators such that the eigenvalues of $\Lambda_s(0)$, $\Lambda_c$ and $\Lambda_u(0)$ have modulus less than $1$, equal to 1 and greater than $1$, respectively.
Basically, Takens Theorem can be used to obtain the normal form of a partially hyperbolic system (i.e., a system with a center direction) via a $C^k$ transformation. This is important because the Takens normal form can help one to
linearize a partially hyperbolic system along its center manifold and can also help to obtain the existence of a  smooth invariant foliation near the center manifold. Consequently, Takens Theorem has wide applications to the studies of bifurcations (\cite{NPT-IHES83, Rob-Nonlin89}), nonlinear oscillations (\cite{Chi-Nonlin10}), Hilbert's 16th problem (\cite{DR-CPAA09}), the $N$-body problem (\cite{D-Nonlin21}), boundary value
problems (\cite{DS-TAMS95}), random tree models (\cite{CDHM-JSP23}), etc.

Due to the importance of Takens Theorem, providing its extensions and generalizations  has been an active area of research since the 1990s. On one hand, $C^\infty$ versions of Takens Theorem were obtained when the hyperbolic linear part is independent of the center variable or the hyperbolic direction is purely contractive (or expansive) (see, e.g., \cite{B-TAMS96,DR-QTDS10,D-QTDS21}). On the other hand, Takens Theorem was extended to the setting of random dynamical systems (\cite{LL-DCDS16}). Most of these works employed the so-called ``path method" or ``homotopy method" developed by Mather (\cite{Ma-AnnMath68}). In this method, a crucial intermediate result is to show that two systems having identical derivatives up to the $N$-th order on their common center manifold are $C^k$ conjugated. This result is known as   the Belitskii-Samavol Theorem (\cite{IL-AMS98}). After 2000, $C^\infty$ versions of Belitskii-Samavol Theorem were obtained in \cite{BK-JDDE02, Ray-ETDS09} on $\mathbb{R}^d$ and on Banach spaces, respectively.

In the case of  nonautonomous dynamics on $\mathbb{R}^d$, we are concerned with the system
\begin{equation}\label{DNDS}
x_{n+1}=F_n(x_n):=A_n x_n + f_n(x_n),\quad \forall n\in \mathbb{Z},
\end{equation}
where $A_n:\mathbb{R}^d\to \mathbb{R}^d$ are invertible matrices and $f_n:\mathbb{R}^d\to \mathbb{R}^d$ are smooth mappings.
The smooth conjugacy of two nonautonomous systems $(F_n)_{n\in\mathbb{Z}}$ and $(G_n)_{n\in\mathbb{Z}}$ is
a sequence $(\Psi_n)_{n\in\mathbb{Z}}$ of diffeomorphisms with bounded derivatives near $0$ such that
$$
\Psi_{n+1}\circ F_n=G_n\circ \Psi_n, \quad  \forall n\in \mathbb{Z}.
$$
In the uniformly hyperbolic case, i.e., when $(A_n)_{n\in\mathbb{Z}}$ admits a uniform exponential dichotomy, the nonautonomous version of Poincar\'e Theorem can be found in \cite{Siegm-JDE02,WL-JDE08}. In 2019, the nonautonomous version of
Sternberg Theorem was given in \cite[Theorem 5]{CDS-JDDE19}. The conditions of~\cite[Theorem 5]{CDS-JDDE19} are expressed in terms of the so-called
{\it dichotomy spectrum}  $\Sigma(A_n)$
which consists of all $\lambda >0$ with the property that the sequence $(\lambda^{-1}A_n)_{n\in \mathbb Z}$ does not admit an exponential dichotomy. Provided that $\sup_{n\in \mathbb Z}\|A_n\|<+\infty$ and $\sup_{n\in \mathbb Z}\|A_n^{-1}\|<+\infty$, we have that
\[
\Sigma(A_n)=\bigcup_{i=1}^r[a_i,b_i] \quad   \text{for some $r\le d$,}
\]
where
\[
0<a_1 \le b_1 <\ldots <a_r\le b_r.
\]
Provided that $1\notin \Sigma(A_n)$, it is proved in~\cite[Theorem 5]{CDS-JDDE19} that  for any $k\in\mathbb{N}$, there is an integer $N\in\mathbb{N}$ such that if
the non-resonance condition
\begin{align}\label{S-NC}
\left([a_1,b_1]^{t_1}\cdots[a_r,b_r]^{t_r}\right)\cap [a_j,b_j]=\emptyset,\quad\forall j=1,\ldots ,r
\end{align}
holds for integers $t_i\ge 0$ such that $2\le t_1+\ldots+t_r\le N$, where $[a,b]^s:=[a^s,b^s]$ and $[a,b]\cdot[c,d]:=[ac,bd]$ for $a,b,c,d,s\in\mathbb{R}$,
then the $C^N$ system $(F_n)_{n\in\mathbb{Z}}$ can be locally $C^k$ linearized.  For different smooth linearization results dealing with nonautonomous dynamics, we refer to~\cite{DZZ,DZZ-PLMS20}. In addition, we note that the study of nonautonomous linearization goes back to the seminal work of Palmer~\cite{Palmer}.

The main objective of the present  paper is to extend the Takens Theorem to the nonautonomous case when the linear part $(A_n)_{n\in\mathbb{Z}}$ of~\eqref{DNDS} admits a uniform exponential trichotomy.  Provided that $(A_n)_{n\in \mathbb Z}$ does not admit an exponential dichotomy, its dichotomy
spectrum has the form
$
\Sigma(A_n)=\bigcup_{i=1}^r[a_i,b_i]\cup [\nu_-, \nu_+]
$
with
\[
0<a_1\le b_1 <\ldots a_l\le b_l<\nu_{-}\le 1\le \nu_+<a_{l+1}\le b_{l+1}<\ldots a_r\le b_r,
\]
for some $r<d$ and $l\in \{1, \ldots, r-1\}$.
Then, we have the following main result of our paper.

\noindent {\bf Main Theorem}.
{\it For any $k\in\mathbb{N}$, there is  $N\in\mathbb{N}$ such that if the non-resonance conditions
\begin{align}\label{NR-Main}
\begin{split}
&([a_1,b_1]^{q_1}\cdots[a_r,b_r]^{q_r}\cdot[\nu_-,\nu_+]^{N})\cap [\nu_-,\nu_+]=\emptyset,
\\
&([a_1,b_1]^{t_1}\cdots[a_r,b_r]^{t_r}\cdot[\nu_-,\nu_+]^{N})\cap [a_j,b_j]=\emptyset,\quad \forall j=1,\ldots, r,
\end{split}
\end{align}
hold for integers $q_i,t_i\ge 0$ such that $1\le q_1+\ldots+q_r\le N$ and $2\le t_1+\ldots+t_r\le N$, then the $C^{N+1}$ system  $(F_n)_{n\in\mathbb{N}}$ is locally $C^k$ conjugated to
$$
(x_s,x_c,x_u)^T\mapsto(A_n^s(x_c)x_s, A_n^cx_c+f_n^c(x_c),A_n^u(x_c)x_u)^T,
$$
where $A_n^s(x_c)$ and $A_n^u(x_c)$ are linear operators, which
are $C^{N-1}$ in $x_c$, and $f_n^c$ is a $C^N$ map.
}

\noindent We note that the above theorem includes the hyperbolic case (when the center direction does not exist) in which we interpret~\eqref{NR-Main} as~\eqref{S-NC}. Moreover, \eqref{NR-Main} reduces to \eqref{S-NR} in the autonomous partially hyperbolic case (when  $\nu_-=\nu_+=1$).  One of the  difficulties in the proof of the above theorem is solving homological equations for the normal form theory which now involve a center variable (see \eqref{elimi-equa-chan} and \eqref{elimi-equa-proj} below), while another difficulty is finding the dichotomy spectrum of a certain matrix cocycle that is block lower triangular (see \eqref{block-low}). In order to overcome those difficulties, in comparison with \eqref{S-NR}, we need an additional term $[\nu_-,\nu_+]^{N}$ in \eqref{NR-Main} to guarantee certain spectral gap conditions (see \eqref{SG-cond-1} and \eqref{new-gap}).

The paper is organized as follows.
In Section 2, we recall the notions of exponential dichotomy and exponential trichotomy as well as that of the dichotomy spectrum. In Section 3, we give a center manifold result in the nonautonomous case   by lifting the nonautonomous system on $\mathbb{R}^d$ to the autonomous one which acts on a certain Hilbert space.
Then, we prove our main theorem in Section 4 by first establishing the nonautonomous version of Belitskii-Samavol Theorem by using the ``lifting method". Finally, in the Appendix, we give the Belitskii-Samavol Theorem on Hilbert spaces for autonomous dynamics. 

\section{Preliminaries}
\setcounter{equation}{0}
\noindent Let $\mathbb R^d$ denote the $d$-dimensional Euclidean space equipped with the Euclidean  norm $\| \cdot \|$. The associated operator norm on the space of linear operators on $\mathbb R^d$ will also be denoted by $\| \cdot \|$.
Let $(A_n)_{n \in \mathbb{Z}}$ be a two-sided sequence of invertible linear operators acting on $\mathbb R^d$.
We define the associated \emph{linear cocycle} by
\begin{eqnarray}
{\mathcal A}(m,n):=\begin{cases}
A_{m-1} \cdots A_n & \text{if $m > n$,} \\
{\Id} & \text{if $m=n$,}\\
A_m^{-1} \cdots A_{n-1}^{-1} & \text{if $m<n$,}
\end{cases}
\label{AAA}
\end{eqnarray}
where $\Id$ denotes the identity operator on $\mathbb R^d$.
We recall the notions of exponential dichotomy and exponential trichotomy as follows.
\begin{df}\label{ED}
The linear system $(A_n)_{n \in \mathbb{Z}}$ admits an exponential dichotomy if there exist constants $K>0$,
$0<\alpha_+<1<\beta_-$
and projections $\Pi^s_n, \Pi^u_n$, $n\in \mathbb Z$ on $\mathbb R^d$ such that the following conditions hold:
\begin{itemize}
\item for every $n\in \mathbb Z$,
$
\Pi^s_n+\Pi^u_n=\Id;
$
\item for every $n\in \mathbb Z$,
$\Pi^s_{n+1}A_n=A_n\Pi^s_n$;
\item for $m\ge n$,
\begin{align*}
 \| \mathcal A(m,n)\Pi^s_n \| \le K\alpha_+^{m-n},\qquad \| \mathcal A(n,m)\Pi^u_m \| \le K\beta_-^{n-m}.
\end{align*}
\end{itemize}
In addition, if there exist $\alpha_{-}\le \alpha_+$ and $\beta_{+}\ge \beta_-$ such that
\begin{align*}
 \| \mathcal A(n,m)\Pi^s_m \| \le K\alpha_-^{n-m},\qquad \|\mathcal A(m,n)\Pi^u_n\| \le K\beta_+^{m-n}
\end{align*}
hold for $m\ge n$, then $(A_n)_{n \in \mathbb{Z}}$ is said to admit a strong exponential dichotomy.
\end{df}

\begin{df}\label{ET}
The linear system $(A_n)_{n \in \mathbb{Z}}$ admits an exponential trichotomy if there exist constants
\[
K>0,\quad 0< \mu_+<\lambda_{-}\le 1\le  \lambda_+<\varrho_-,
\]
and projections $\Pi^s_n, \Pi^c_n, \Pi^u_n$, $n\in \mathbb Z$ on $\mathbb R^d$ such that the following conditions hold:
\begin{itemize}
\item for every $n\in \mathbb Z$,
\begin{equation}\label{pro1}
\Pi^s_n+\Pi^c_n+\Pi^u_n=\Id;
\end{equation}
\item for every $n\in \mathbb Z$,
\begin{equation}\label{pro2}
\Pi^s_{n+1}A_n=A_n\Pi^s_n, \quad \Pi^c_{n+1}A_n=A_n \Pi^c_n, \quad \Pi^u_{n+1}A_n=A_n \Pi^u_n;
\end{equation}
\item for $m\ge n$,
\begin{alignat}{2}
 &\| \mathcal A(m,n)\Pi^s_n \| \le K\mu_+^{m-n},&\qquad & \| \mathcal A(n,m)\Pi^u_m\| \le K\varrho_-^{n-m},
 \label{tric1}
 \\
 &\|\mathcal A(m,n)\Pi^c_n \| \le K\lambda_+^{m-n},&\qquad&\|\mathcal A(n,m)\Pi^c_m\| \le K\lambda_-^{n-m}.\label{tric2}
\end{alignat}
\end{itemize}
In addition, if there exist $0<\mu_{-}\le \mu_+$ and $\varrho_+\ge \varrho_{-}$ such that
\begin{align}
 \| \mathcal A(n,m)\Pi^s_m \| \le K\mu_-^{n-m},\qquad \|\mathcal A(m,n)\Pi^u_n\| \le K\varrho_+^{m-n}
 \label{tric3}
\end{align}
hold for $m\ge n$, then $(A_n)_{n \in \mathbb{Z}}$ is said to admit a strong exponential trichotomy.
\end{df}
\begin{rmk}
It is easy to show that a sequence $(A_n)_{n\in \mathbb Z}$ which admits an exponential dichotomy {\rm (}resp. trichotomy{\rm )} admits a strong exponential dichotomy {\rm (}resp. trichotomy{\rm )} if and only if
\begin{equation}\label{bnd}
    \sup_{n\in \mathbb Z}\|A_n\|<+\infty \quad \text{and} \quad \sup_{n\in \mathbb Z}\|A_n^{-1}\|<+\infty.
\end{equation}
Indeed, suppose for example that $(A_n)_{n\in \mathbb Z}$ admits a strong exponential trichotomy. Then, for $m\ge n$, \eqref{pro1}, \eqref{tric1}-\eqref{tric3} give that
\begin{align}
\|\mathcal A(m,n)\| &= \|\mathcal A(m,n)(\Pi^s_n + \Pi^c_n+\Pi^u_n)\|\notag
\\
&\le \|\mathcal A(m,n) \Pi^s_n \| + \|\mathcal A(m,n) \Pi^c_n \|+\|\mathcal A(m,n) \Pi^u_n \|
\notag\\
&\le 3 K \varrho_+^{m-n},
\label{est1}
\end{align}
and
\begin{equation}\label{est2}
\begin{split}
\|\mathcal A(n,m)\| &= \|\mathcal A(n,m)(\Pi^s_n + \Pi^c_n+\Pi^u_n)\|
\\
&\le \|\mathcal A(n,m) \Pi^s_n \| + \|\mathcal A(n,m) \Pi^c_n \|+\|\mathcal A(n,m) \Pi^u_n \|
\\
&\le 3 K \mu_-^{n-m}.
\end{split}
\end{equation}
Clearly, \eqref{est1} and~\eqref{est2} imply~\eqref{bnd}. The converse implication is  obvious.
\end{rmk}

\begin{rmk}\label{1100}
Assume that a sequence $(A_n)_{n\in \mathbb Z}$ admits a strong exponential trichotomy. Then, it is easy to verify the following:
\begin{itemize}
\item for $\gamma \in (0, \mu_{-})$, the sequence $(\gamma^{-1}A_n)_{n\in \mathbb Z}$ admits a strong exponential dichotomy with respect to projections $\Pi_n^{\gamma,s}=0$ and $\Pi_n^{\gamma,u}=\Id$ for $n\in \mathbb Z$;
\item for $\gamma \in (\mu_+, \lambda_{-})$, the sequence $(\gamma^{-1}A_n)_{n\in \mathbb Z}$ admits a strong exponential dichotomy with respect to projections $\Pi_n^{\gamma,s}=\Pi_n^s$ and $\Pi_n^{\gamma,u}=\Pi_n^c+\Pi_n^u$ for $n\in \mathbb Z$;
\item for $\gamma \in (\lambda_+, \varrho_{-})$, the sequence $(\gamma^{-1}A_n)_{n\in \mathbb Z}$ admits a strong exponential dichotomy with respect to projections $\Pi_n^{\gamma,s}=\Pi_n^s+\Pi_n^c$ and $\Pi_n^{\gamma,u}=\Pi_n^u$ for $n\in \mathbb Z$;
\item for $\gamma \in (\varrho_+, \infty)$, the sequence $(\gamma^{-1}A_n)_{n\in \mathbb Z}$ admits a strong exponential dichotomy with respect to projections $\Pi_n^{\gamma,s}=\Id$ and $\Pi_n^{\gamma,u}=0$ for $n\in \mathbb Z$.
\end{itemize}
\end{rmk}

Let us fix an arbitrary sequence $(A_n)_{n\in \mathbb Z}$ that admits a strong exponential trichotomy and  which does not admit an exponential dichotomy.
We  define the {\it dichotomy spectrum} of $(A_n)_{n \in \mathbb{Z}}$ by
\begin{align}\label{def-sig}
\Sigma(A_n):=\Big\{&\gamma\in (0,\infty):
(\gamma^{-1}A_n)_{n\in\mathbb{Z}} ~\mbox{admits no exponential dichotomy}\Big\}.
\end{align}
According to the Spectral Theorem (\cite{AS00}) and Remark~\ref{1100}, we know that there is  $r\le d-1$ such that
\begin{align}\label{SA}
\Sigma=\Sigma(A_n)=\bigcup_{i=1}^r[a_i,b_i]\cup[\nu_-,\nu_+]
\subset [\mu_-,\mu_+]\cup[\lambda_-,\lambda_+]\cup[\varrho_-,\varrho_+],
\end{align}
where
\[
\mu_{-}\le a_1\le b_1<\ldots< a_\ell \le b_\ell \le \mu_+, \quad  \lambda_{-}\le \nu_{-}\le 1 \le \nu_+\le \lambda_+
\]
and
\[
\varrho_{-}\le a_{\ell+1}\le b_{\ell+1}<\ldots<a_r\le b_r\le \varrho_+,
\]
with a certain $\ell\in \{ 0,\ldots,r\}$. In the sequel we will assume that $r\ge 1$, i.e., that $\Sigma$ does not consist only of $[\nu_{-}, \nu_+]$.  Observe that in~\eqref{SA} we have supposed that $\Sigma \cap [\lambda_{-}, \lambda_+]$ consists of a single interval. It can be easily seen that this can be assumed without any  loss of generality, as it can be achieved by redefining projections (and corresponding constants) in the trichotomy of the sequence $(A_n)_{n\in \mathbb Z}$.
Moreover,
remark that when $\ell=0$ (or $\ell=r$), the system has only expansive (or contractive) and center directions, which is a special case of the partial hyperbolicity. In what follows, we do not discuss the case of $\ell\in \{0,r\}$ because it can be reduced to the case of $\ell\in \{ 1,\ldots,r-1\}$
by adding some dummy variables to the expansive (or contractive) part.
Finally, we note that in \eqref{SA} the union $\bigcup_{i=1}^r[a_i,b_i]\subset [\mu_-,\mu_+]\cup[\varrho_-,\varrho_+]$ corresponds to the hyperbolic part of $\Sigma$ as  $1\notin \bigcup_{i=1}^r[a_i,b_i]$, while the interval $[\nu_-,\nu_+]$ corresponds to the center part of $\Sigma$ as $1\in [\nu_-,\nu_+]\subset [\lambda_-,\lambda_+]$.

Furthermore,
according to \cite{Siegm-JDEA02},
we may assume that $(A_n)_{n\in \mathbb Z}$ is block-diagonalized, i.e.,
$$
A_n={\rm diag}(A_n^1,\ldots ,A_n^\ell, A_n^c,A_n^{\ell+1},\ldots , A_n^r)=:{\rm diag}(A_n^s, A_n^c,A_n^u),
$$
where \begin{equation}\label{red}\Sigma(A_n^i)=[a_i,b_i] \  \text{for $i=1,\ldots ,r$ and} \quad  \Sigma(A_n^c)=[\nu_-,\nu_+].
\end{equation}The above block-diagonal form of operators $A_n$ generates the
corresponding decomposition of $\mathbb R^d$ given by
\begin{align}\label{DS}
\mathbb{R}^d=X_1\oplus\cdots\oplus X_\ell\oplus X_c\oplus X_{\ell+1}\oplus\cdots \oplus X_r.
\end{align}
For convenience, denote
\begin{align*}
&X_s:=X_1\oplus\ldots \oplus X_\ell,\quad
X_u:=X_{\ell+1}\oplus \ldots \oplus X_r,\quad X_{su}:=X_s\oplus X_u,
\\
&x_s:=x_1+\ldots + x_\ell\in X_s,\qquad x_u:=x_{\ell+1}+\ldots+x_r\in X_u
\end{align*}
for $x_i\in X_i$, $i=1,\ldots ,r$,
and denote
$$
\pi_ix:=x_i, \quad\pi_sx:=x_s,\quad \pi_cx:=x_c,\quad\pi_ux:=x_u,\quad \pi_{su}:=\pi_s+\pi_u,
$$
for $x=x_s+x_c+x_u\in \mathbb R^d$.
Taking into account~\eqref{red}, it follows that for a  sufficiently small $\varsigma>0$ there exists $C>0$ such that
\begin{align}\label{new-ET}
\begin{split}
&\|{\mathcal A}(m,n)\pi_ix\|\le C(b_i+\varsigma)^{m-n}\|\pi_ix\|,\quad m\ge n,
\\
&\|{\mathcal A}(m,n)\pi_ix\|\le C(a_i-\varsigma)^{m-n}\|\pi_ix\|,\quad m\le n,
\\
&\|{\mathcal A}(m,n)\pi_cx\|\le C(\nu_+ +\varsigma)^{m-n}\|\pi_cx\|,\quad m\ge n,
\\
&\|{\mathcal A}(m,n)\pi_cx\|\le C(\nu_- -\varsigma)^{m-n}\|\pi_cx\|,\quad m\le n,
\end{split}
\end{align}
 for $i=1, \ldots ,r$.

\begin{rmk}
Since another formula $\mathbb{R}^d=X_1\times\cdots\times X_\ell\times X_c\times X_{\ell+1}\times\cdots \times X_r$ is essentially the same as (\ref{DS}), in what follows, we will use either of the two formulae according to our needs. Namely, we will selectively use $x_1+\cdots+x_r$ or $(x_1,...,x_r)^T$.
\end{rmk}

\begin{rmk}\label{simpl}
In order to simplify the notation, in the sequel we will assume that estimates in~\eqref{new-ET} hold with $\varsigma=0$. This will not cause any issues since every strict inequality {\rm (}see for example \eqref{NR-cond-1}-\eqref{NR-cond-2}{\rm )} for $a_i,b_i,\nu_-$ and $\nu_+$ holds if and only if it holds for $a_i-\varsigma,b_i+\varsigma,\nu_--\varsigma$ and $\nu_++\varsigma$ with a  sufficiently  small $\varsigma>0$.
 \end{rmk}

Next, set
\begin{equation}\label{spaceY}
 Y:=\Big{\{} \mathbf x=(x_n)_{n\in \mathbb{Z}} \subset \mathbb R^d: \sum_{n\in\mathbb{Z}}\|x_n\|^2 <\infty \Big{\}},
\end{equation}
which is a Banach space when equipped with the norm
$
\|\mathbf x\|:=(\sum_{n\in\mathbb{Z}}\|x_n\|^2)^{\frac{1}{2}}.
$
For the sake of simplicity we slightly abuse the notation by denoting the norm on $Y$ and $\mathbb R^d$ by the same symbol.
In addition,  $Y$ is actually a Hilbert space with respect to the scalar product given by
\begin{equation}\label{in-pro}
\langle \mathbf x, \mathbf y\rangle_Y:=\sum_{n\in \mathbb Z} \langle x_n, y_n\rangle \quad \mathbf x=(x_n)_{n\in \mathbb Z}, \ \mathbf y=(y_n)_{n\in \mathbb Z}\in Y,
\end{equation}
where $\langle \cdot, \cdot \rangle$ is the Euclidean scalar  product on $\mathbb R^d$.



This enables us to define a  linear operator $\mathbb A \colon Y \to Y$ by
\begin{equation}\label{eq:A}
(\mathbb A \mathbf x)_n:=A_{n-1}x_{n-1}, \quad \forall\mathbf x=(x_n)_{n\in \mathbb{Z}} \in Y,~ \forall n \in \mathbb{Z}.
\end{equation}
It follows easily from~\eqref{bnd} that $\mathbb A$ is bounded. Indeed,
\[
\| \mathbb A \mathbf x\|^2=\sum_{n\in \mathbb Z}\|A_{n-1}x_{n-1}\|^2 \le \left (\sup_{n\in \mathbb Z}\|A_n\|\right )^2\sum_{n\in \mathbb Z}\|x_{n-1}\|^2 =\left (\sup_{n\in \mathbb Z}\|A_n\|\right )^2 \|\mathbf x\|^2,
\]
for each $\mathbf x=(x_n)_{n\in \mathbb Z}\in Y$, which yields that $\mathbb A$ is well-defined and bounded. Moreover, $\mathbb A$
is invertible and its inverse is given by
\[
(\mathbb A^{-1} \mathbf x)_n=A_{n}^{-1}x_{n+1}, \quad \forall\mathbf x=(x_n)_{n\in \mathbb{Z}} \in Y,~ \forall n \in \mathbb{Z}.
\]
Indeed, using the second assumption in~\eqref{bnd}, we can easily conclude that the operator defined by the above equation is well-defined and bounded. Next,
\[
(\mathbb A \mathbb A^{-1}\mathbf x)_n=A_{n-1} (\mathbb A^{-1} \mathbf x)_{n-1}=A_{n-1}A_{n-1}^{-1}x_n=x_n,
\]
for $n\in \mathbb Z$ and $\mathbf x=(x_n)_{n\in \mathbb Z}\in Y$. Thus, $\mathbb A \mathbb A^{-1}=\Id_Y$, where $\Id_Y$ denotes the identity operator on $Y$. Similarly, $\mathbb A^{-1}\mathbb A=\Id_Y$, which yields that $\mathbb A^{-1}$ is the inverse of $\mathbb A$.
\begin{rmk}
The reason for considering space $Y$ given by~{\rm \eqref{spaceY}} is that it has a  structure of a Hilbert space. This is effectively used in the proof of Theorem~{\rm \ref{BS-thm-nonauto}}, which in turn relies on Theorem~\ref{ck-conjuga}. One can observe that the proof of Theorem~\ref{ck-conjuga} uses the existence of the so-called smooth bump functions, which are known to exist on Hilbert spaces and whose existence in general fails to hold on Banach spaces.

One also observes that the spectrum of $\mathbb A$ is independent of whether we consider $\mathbb A$ as an operator on $Y$ or as an operator on the space
\[
Y_\infty:=\left \{ \mathbf x=(x_n)_{n\in \mathbb Z}\subset \mathbb R^d: \  \| \mathbf x\|_\infty:=\sup_{n\in \mathbb Z}\|x_n\|<+\infty \right \}.
\]
Indeed, this follows from the  results in~\cite{BDV} which imply that the following three conditions are equivalent:
\begin{enumerate}
\item[(a)] a sequence $(B_n)_{n\in \mathbb Z}$ of  linear operators on $\mathbb R^d$ admits an exponential dichotomy;
\item[(b)] for each $\mathbf y=(y_n)_{n\in \mathbb Z}\in Y$, there exists a unique $\mathbf x=(x_n)_{n\in \mathbb Z}\in Y$ such that
\begin{equation}\label{admiss}
x_{n+1}-B_n x_n=y_{n+1} \quad n\in \mathbb Z;
\end{equation}
\item[(c)]for each $\mathbf y=(y_n)_{n\in \mathbb Z}\in Y_\infty$, there exists a unique $\mathbf x=(x_n)_{n\in \mathbb Z}\in Y_\infty$ such that~\eqref{admiss} holds.
\end{enumerate}
Let $\mathbb A_\infty$ denote the operator $\mathbb A$ viewed as an operator on $Y_\infty$. Note that $\mathbb A_\infty$ is also invertible.
Then, from $(a)\iff (b)\iff (c)$, we have that
for $\lambda \in \mathbb R\setminus \{0\}$,
\[
\begin{split}
\lambda \notin \sigma(\mathbb A_\infty) &\iff  \  \text{the operator $\mathbf x \mapsto (x_n-\lambda^{-1}A_{n-1}x_{n-1})_{n\in \mathbb Z}$ is invertible on $Y_\infty$} \\
&\iff \text{sequence $\left (\frac 1 \lambda A_n \right )_{n\in \mathbb Z}$ admits exponential dichotomy}\\
&\iff \text{the operator $\mathbf x \mapsto (x_n-\lambda^{-1}A_{n-1}x_{n-1})_{n\in \mathbb Z}$ is invertible on $Y$} \\
&\iff \lambda \notin \sigma(\mathbb A).
\end{split}
\]
Since $0$ is not contained in $\sigma(\mathbb A)$ and $\sigma(\mathbb A_\infty)$, we conclude that $\sigma(\mathbb A)=\sigma(\mathbb A_\infty)$. Moreover,
\[
\sigma(\mathbb A)\cap (0, \infty)=\sigma(\mathbb A_\infty)\cap (0, \infty)=\Sigma(A_n).
\]
Since $Y_\infty$ does not have a structure of a Hilbert space, it is more convenient for us to view $\mathbb A$ as an operator on $Y$.

Finally, we note that the conditions of our main result {\rm (}Theorem~{\rm \ref{Thm-Takens})} are stated directly in terms of the sequence $(A_n)_{n\in \mathbb Z}$ and its nonlinear perturbations. Therefore, the use of $\mathbb A$ {\rm (}and the space on which it acts{\rm )} is only of auxiliary nature. It is used to establish the existence of nonautonomous center manifolds by applying known results about the existence of center manifolds for autonomous dynamics. Moreover, it is used to deduce Theorem~{\rm \ref{BS-thm-nonauto}} from Theorem~{\rm \ref{ck-conjuga}}.
\end{rmk}
We define projections $\Pi_{Y}^s$, $\Pi_{Y}^u$ and $\Pi_{Y}^c$ on $Y$ by
\[
(\Pi_{Y}^s \mathbf x)_n:=\Pi^s_nx_n, \quad (\Pi_{Y}^c \mathbf x)_n:=\Pi^c_nx_n, \quad \text{and}\quad (\Pi_{Y}^u \mathbf x)_n:=\Pi^u_nx_n,
\]
for $n\in \mathbb Z$ and $\mathbf x=(x_n)_{n\in \mathbb Z}\in Y$.
One can easily verify that:\begin{itemize}
    \item $\Pi_{Y}^s+\Pi_{Y}^c+\Pi_{Y}^u=\Id;$
    \item $\mathbb A \Pi_{Y}^* = \Pi_{Y}^* \mathbb A$ for $*=s,c,u$;
    \item for $m\ge 0$,
    \begin{align}\label{smA}
    \begin{split}
        &\| \mathbb A^m \Pi_{Y}^s \| \le K \mu_+^m,\qquad \| \mathbb A^{-m} \Pi_{Y}^s \| \le K \mu_-^{-m},
        \\
        & \| \mathbb A^m \Pi_{Y}^c \| \le K \lambda_+^m,\qquad \| \mathbb A^{-m} \Pi_{Y}^c \| \le K \lambda_-^{-m},
        \\
        &\| \mathbb A^m \Pi_{Y}^u \| \le K \varrho_+^m,\qquad \| \mathbb A^{-m} \Pi_{Y}^u \| \le K \varrho_-^{-m}.
        \end{split}
    \end{align}
\end{itemize}
Let us verify the first inequality in~\eqref{smA}. For $\mathbf x=(x_n)_{n\in \mathbb Z}\in Y$ we have that
\[
\begin{split}
\|\mathbb A^m \Pi_{Y}^s \mathbf x\|^2=\sum_{n\in \mathbb Z}\|(\mathbb A^m \Pi_{Y}^s \mathbf x)_n\|^2 &=\sum_{n\in \mathbb Z}\|\mathcal A(n, n-m)\Pi_{n-m}^s x_{n-m}\|^2 \\
&\le K^2  \mu_+^{2m}\sum_{n\in \mathbb Z}\|x_{n-m}\|^2=K^2\mu_+^{2m}\|\mathbf x\|^2,
\end{split}
\]
which yields the desired claim. One can similarly verify all other inequalities in~\eqref{smA}.
In the sequel, by $Y^*$ we will denote the image of the projection $\Pi_{Y}^*$ for $*=s, c, u$.

\section{Center manifolds}\label{CM}
\setcounter{equation}{0}

\noindent This section is devoted to
 the existence of smooth center manifolds of a nonautonomous system whose linear part $(A_n)_{n\in\mathbb{Z}}$ admits a strong exponential trichotomy (but does not admit an exponential dichotomy). Take $N\in \mathbb N$ and consider a sequence $(f_n)_{n\in \mathbb Z}$ of $C^{N+1}$ maps $f_n\colon \mathbb R^d \to \mathbb R^d$.
In what follows, by $d_x^if_n$ we denote the $i$-th derivative of $f_n$ at $x$. Moreover, we will write $d_xf_n$ instead of $d_x^1f_n$. Suppose that:
\begin{itemize}
\item  $(A_n)_{n\in \mathbb Z}$ admits a strong exponential trichotomy (but does not admit an exponential dichotomy). In particular,
$\Sigma(A_n)$ satisfies \eqref{SA};
\item for $n\in \mathbb Z$,
\begin{equation}\label{zero}
f_n(0)=0 \quad  \text{and} \quad d_0f_n=0;
\end{equation}
\item for $n\in \mathbb Z$, $2\le i\le  N+1$ and $x\in \mathbb R^d$,
\begin{equation}\label{small-Lip}
\|d_xf_n\|\le \delta\quad \text{and} \quad\|d^i_xf_n\| \le M,
\end{equation}
where $\delta, M>0$ are some constants. In the sequel, we will assume that $\delta$ can be made sufficiently small.
\end{itemize}
Let $\Ima \Pi^*_n$ denote the image of $\Pi^*_n$
for $*=s,c,u$.  Moreover, set
\[
F_n:=A_n+f_n \colon \mathbb R^d \to \mathbb R^d, \quad n\in \mathbb Z.
\]
Then, we have the following main result of this section.

\begin{prop}\label{center}
Let $(F_n)_{n\in\mathbb{Z}}$ be given above and suppose that the spectral gap condition
\begin{align}\label{gap-cond}
b_\ell<\nu_-^N,\qquad \nu_+^N< a_{\ell+1},
\end{align}
holds. Then, there exists a sequence of $C^N$ maps $\phi_n\colon \Ima \Pi^c_n \to \Ima \Pi^s_n \oplus \Ima \Pi^u_n $, $n\in \mathbb Z$, such that for $n\in \mathbb Z$:
\begin{itemize}
\item $\phi_n(0)=0$, $d_0\phi_n=0$ and $\|d_{x}^i\phi_n\|\le \tilde M$ near $x=0$   for  $i=2,\ldots ,N$, with respect to some constant $\tilde M>0$ independent on $n$ and $i$;
\item $F_n(\Gamma_n)=\Gamma_{n+1}$, where
$
\Gamma_n:=\{x_c+\phi_n(x_c): \ x_c\in \Ima \Pi^c_n\}
$
is the center manifold;
\item for any $x_c\in \Ima \Pi^c_n$ and $\gamma_1\in (\mu_+,\nu_-^N),\gamma_2\in (\nu_+^N,\varrho_-)$, we have
\begin{equation}\label{cm1}
\sup_{m> n}\left\{\gamma_2^{-(m-n)}\|(F_{m-1}\circ \ldots \circ F_n)(x_c+\phi_n(x_c))\|\right\}<\infty,
\end{equation}
and
\begin{equation}\label{cm2}
\sup_{m<n}\left\{\gamma_1^{-(m-n)}\|(F_m^{-1}\circ \ldots \circ F_{n-1}^{-1})(x_c+\phi_n(x_c))\|\right\}<\infty.
\end{equation}
\end{itemize}
\end{prop}

Remark that~\eqref{gap-cond} holds if
\begin{equation}\label{tricc}
\mu_+<\lambda_-^N \quad \text{and} \quad \lambda_+^N< \varrho_{-},
\end{equation}
which is expressed purely in terms of constants associated with the trichotomy of the sequence $(A_n)_{n\in \mathbb Z}$.

\begin{proof}[Proof of Proposition {\rm \ref{center}}]
We define $\mathbb F\colon Y \to Y$ by
\begin{equation}\label{F}
(\mathbb F (\mathbf x))_n:=F_{n-1}(x_{n-1}):=A_{n-1}x_{n-1}+f_{n-1}(x_{n-1}), \quad \forall n\in \mathbb Z,
\end{equation}
for $\mathbf x=(x_n)_{n\in \mathbb Z}\in Y$, and recall that $\mathbb A \colon Y\to Y$ is given by~\eqref{eq:A}. Let us first note that $\mathbb F$ is well-defined. Indeed, for $\mathbf x=(x_n)_{n\in \mathbb Z}\in Y$ we have that
\[
\sum_{n\in \mathbb Z}\|(\mathbb F (\mathbf x))_n\|^2 \le  \left (\sup_{n\in \mathbb Z}\|A_n\|+\delta \right)^2\sum_{n\in \mathbb Z}\|x_{n}\|^2<+\infty
\]
by (\ref{bnd}), which yields that $\mathbb F(\mathbf x)\in Y$.
\begin{lm}\label{lemm}
$\mathbb F$ is $C^N$, $\mathbb F(\boldsymbol 0)=\boldsymbol 0$ and $d_{\boldsymbol 0}\mathbb F=\mathbb A$. Moreover,
\begin{equation}\label{LIP}
\|d_{\bf x}(\mathbb F-\mathbb A)\| \le  \delta, \quad \|d^i_{\bf x} \mathbb F\| \le  M,\quad \forall {\bf x}\in Y,
\end{equation}
for $2\le i \le N$,
where $\delta>0$ and $M>0$ are  as in \eqref{small-Lip}.
\end{lm}

\begin{proof}[Proof of Lemma {\rm \ref{lemm}}]
We first claim that $\mathbb F$ is differentiable and that
\begin{equation}\label{1st-deri}
d_{\mathbf  x}\mathbb F (\boldsymbol{\xi}) = ( A_{n-1}\xi_{n-1}+ d_{x_{n-1}}f_{n-1} \xi_{n-1})_{n\in\mathbb{Z}},
\end{equation}
for  $\mathbf x=(x_n)_{n\in \mathbb Z},
\boldsymbol{\xi}=(\xi_n)_{n\in \mathbb Z}\in Y$. We define $L\colon Y\to Y$ by
\[
L\boldsymbol{\xi}=( A_{n-1}\xi_{n-1}+ d_{x_{n-1}}f_{n-1} \xi_{n-1})_{n\in\mathbb{Z}}.
\]
It is easy to verify that $L$ is a bounded linear operator on $Y$.  Take $\mathbf y\in Y$.
Note that
\[
\begin{split}
 (\mathbb F(\mathbf x+\mathbf y)-\mathbb F(\mathbf x)-L\mathbf y)_n
&=\int_0^1 (d_{x_{n-1}+ty_{n-1}}f_{n-1}y_{n-1}-d_{x_{n-1}}f_{n-1}y_{n-1})\, dt,
\end{split}
\]
which together with Mean Value Theorem and~\eqref{small-Lip} implies that
\[
\begin{split}
\| (\mathbb F(\mathbf x+\mathbf y)-\mathbb F(\mathbf x)-L\mathbf y)_n\| &\le \int_0^1 \|d_{x_{n-1}+ty_{n-1}}f_{n-1}y_{n-1}-d_{x_{n-1}}f_{n-1}y_{n-1}\|\, dt \\
&\le \left(\sup_{x\in\mathbb R^d} \| d^2_xf_{n-1}\| \right)\|y_{n-1}\|^2
\\
&\le M\|y_{n-1}\|^2,
\end{split}
\]
for $n\in \mathbb Z$. Hence,
\[
\|\mathbb F(\mathbf x+\mathbf y)-\mathbb F(\mathbf x)-L\mathbf y\|\le M\|\mathbf y\|^2.
\]
This readily implies that~\eqref{1st-deri} holds.  In particular, we have that $d_{\mathbf 0} \mathbb F=\mathbb A$.

Next, we further prove that for each $2\le i \le N+1$
\begin{equation}\label{high-deri}
d_{\mathbf  x}^i\mathbb{F}(\boldsymbol{\xi}^1,\ldots, \boldsymbol{\xi}^i) = \left( d^i_{x_{n-1}} f_{n-1} (\xi^1_{n-1},\ldots,\xi^i_{n-1})\right)_{n\in\mathbb{Z}},
\end{equation}
for $\mathbf x=(x_n)_{n\in \mathbb Z}, \boldsymbol{\xi}^j= (\xi^j_n)_{n\in\mathbb{Z}}\in Y$ with $2\le j\le i$. This will be proved by induction. When $i=2$, let  $\mathbb D^2 \colon Y \times Y \to Y$ be a multilinear operator defined by
\begin{equation*}
 \mathbb D^2 ( \boldsymbol{\xi}^1, \boldsymbol{\xi}^2 ):= \left( d^2_{x_{n-1}} f_{n-1}(\xi^1_{n-1},\xi^2_{n-1})\right)_{n\in\mathbb{Z}}
\end{equation*}
for $\boldsymbol{\xi}^1, \boldsymbol{\xi}^2\in Y$.
 Observe that the second inequality of \eqref{small-Lip} (with $i=2$) implies that
$
\| \mathbb D^2 (\boldsymbol{\xi}^1, \boldsymbol{\xi}^2)\| \le M\|\boldsymbol{\xi}^1\| \|\boldsymbol{\xi}^2\|,
$
and therefore $\mathbb D^2$ is bounded. Moreover, by the second inequality of \eqref{small-Lip} (with $i=3$) again and by the Mean Value Theorem, we deduce that
\begin{align}\label{2nd-der-1}
&\|d_{\mathbf  x+\mathbf y}\mathbb F(\boldsymbol{\xi}) - d_{\mathbf x}\mathbb F(\boldsymbol{\xi})- \mathbb D^2(\boldsymbol{\xi}, \mathbf y)\|
\nonumber\\
&=\Big\| \Big ( d_{x_{n-1}+y_{n-1}} f_{n-1}(\xi_{n-1}) - d_{x_{n-1}} f_{n-1}(\xi_{n-1})  - d^2_{x_{n-1}} f_{n-1}(\xi_{n-1}, y_{n-1}) \Big )_{n\in\mathbb Z} \Big \|
\nonumber\\
 &\le \left(\sum_{n\in\mathbb Z} \int_0^1 \| (d^2_{x_{n-1}+ty_{n-1}} f_{n-1}-d^2_{x_{n-1}} f_{n-1})(\xi_{n-1}, y_{n-1})\|^2\, dt\right)^{\frac{1}{2}}
\nonumber\\
 & \le M \| \mathbf  y\|^2 \|\boldsymbol{\xi}\|,
\end{align}
for $\mathbf  y=(y_n)_{n\in \mathbb Z},~ \boldsymbol{\xi}=(\xi_n)_{n\in\mathbb{Z}} \in Y$, which implies that
\begin{align}\label{2nd-der-2}
&\lim_{\bf y\to 0}\frac{\|d_{\bf x+\bf y}\mathbb{F}-d_{\bf x}\mathbb{F}-\mathbb D^2(\cdot, \mathbf y)\|}{\|\bf y\|}
\nonumber\\
&=\lim_{\bf y\to 0}\,\sup_{\boldsymbol{\xi}\in Y\backslash\{0\}}\frac{\|d_{\mathbf  x+\mathbf y}\mathbb F(\boldsymbol{\xi}) - d_{\mathbf x}\mathbb F(\boldsymbol{\xi})- \mathbb D^2(\boldsymbol{\xi}, \mathbf y)\|}{\|\bf y\|\|\boldsymbol{\xi}\|}=0.
\end{align}
This implies that $\mathbb F$ is twice differentiable and that the second-order derivative at ${\bf x}$ is given by $d_{{\bf x}}^2F=\mathbb D^2$. Hence, \eqref{high-deri} holds for $i=2$.
Assume now that~\eqref{high-deri} holds for some $2\le i \le N$, and define a multilinear operator $\mathbb D^{i+1} \colon Y^{i+1} \to Y$ by
\[
\mathbb D^{i+1}(\boldsymbol{\xi}^1, \ldots, \boldsymbol{\xi}^{i+1}):=\left( d_{x_{n-1}}^{i+1}f_{n-1} (\xi_{n-1}^1, \ldots, \xi_{n-1}^{i+1}) \right)_{n\in \mathbb Z},
\]
for $\boldsymbol{\xi}^j=(\xi_n^j)_{n\in \mathbb Z}\in Y$ and $2\le j \le i+1$, which is clearly bounded. Then, similarly to \eqref{2nd-der-1}-\eqref{2nd-der-2}, we can show that \eqref{high-deri} holds for $i+1$.
By induction on $i$ we see that \eqref{high-deri} holds for all $2\le i\le N+1$. This shows that $\mathbb F$ is differentiable $N+1$ times and therefore it is $C^N$.

It remains to establish~\eqref{LIP}.
We infer from \eqref{1st-deri} and the first inequality in~\eqref{small-Lip} that
\begin{align}\label{1st-der}
\|d_{\bf x}(\mathbb F-\mathbb A)\|
&= \sup_{\boldsymbol{\xi}\in Y\backslash\{0\}} \frac{\|d_{\bf x}(\mathbb F-\mathbb A)\boldsymbol{\xi}\|}{\|\boldsymbol{\xi}\|}
= \sup_{\boldsymbol{\xi}\in Y\backslash\{0\}} \frac{\|( d_{x_{n-1}}f_{n-1} (\xi_{n-1}))_{n\in\mathbb{Z}}\|}
{\|\boldsymbol{\xi}\|}
\nonumber\\
&=\sup_{\boldsymbol{\xi}\in Y\backslash\{0\}}  \frac{\left(\sum_{n\in\mathbb Z}\|d_{x_{n-1} }f_{n-1} (\xi_{n-1})\|^2\right)^{\frac{1}{2}} }{\|\boldsymbol{\xi}\|}
\nonumber\\
&\le \sup_{\boldsymbol{\xi}\in Y\backslash\{0\}}  \frac{\left(\sum_{n\in\mathbb Z}\delta^2\|\xi_{n-1}\|^2\right)^{\frac{1}{2}} }{\|\boldsymbol{\xi}\|}
=\sup_{\boldsymbol{\xi}\in Y\backslash\{0\}} \frac{\delta\|\boldsymbol{\xi}\|}
{\|\boldsymbol{\xi}\|}
=\delta,
\end{align}
which proves the first inequality of \eqref{LIP}. Moreover, by \eqref{high-deri} and the second inequality of \eqref{small-Lip}, similar to \eqref{1st-der}, we can prove the second inequality of
\eqref{LIP}.
This completes the proof of Lemma \ref{lemm}.
\end{proof}

We continue with the proof of Proposition \ref{center}.
It follows from~\eqref{tricc}, Lemma~\ref{lemm} and the Center Manifold Theorem on  Banach spaces (see e.g. \cite[Theorem 4.1]{Gallay}) that there exists a $C^N$ map $\Phi \colon Y^c\to Y^s\oplus Y^u$ such that
$\Phi(\boldsymbol{0})=\boldsymbol{0}$, $d_{\boldsymbol{0}}\Phi=0$, $\|d_{\bf x}^i\Phi \|\le \tilde M$ near $\bf x=0$ for $i=2,\ldots,N$ with a constant $\tilde M>0$ and $\mathbb F(\Gamma)=\Gamma$, where
\[
\Gamma:=\{{\mathbf x}_c+\Phi({\mathbf x}_c): \ {\mathbf x}_c \in Y^c \}.
\]
In addition, $\Gamma$ is also characterized as
\begin{equation}\label{cm3}
\Gamma=\bigg \{{\bf x}\in Y: \sup_{n\ge 0}\{\gamma_2^{-n}\|\mathbb F^n({\bf x})\|\}<\infty,\quad \sup_{n\ge 0}\{\gamma_1^{n}\| \mathbb F^{-n}({\bf x})\|\}<\infty \bigg \}.
\end{equation}
Next, for any given $n\in \mathbb Z$ and $x_c\in \Ima \Pi^c_n$, define ${\bf v}=(v_m)_{m\in \mathbb Z} \in Y^c$ by
\begin{align}\label{def-vn}
v_n:=x_c \quad {\rm and} \quad v_m:=0 \quad {\rm for}~ m\neq n,
\end{align}
and define $\phi_n:\Ima \Pi^c_n\to \Ima \Pi^s_n \oplus \Ima \Pi^u_n$ by
\begin{align}\label{phi-phi}
\phi_n(x_c):=(\Phi({\bf v}) )_n.
\end{align}
Similarly to \eqref{F} and \eqref{high-deri}, we see that each $\phi_n$, $n\in\mathbb Z$, is $C^N$ such that
\[
d^i_{x_c}\phi_n(\zeta^1, \ldots, \zeta^i)=(d^i_{\bf v} \Phi({ \boldsymbol \zeta}^1, \ldots, { \boldsymbol \zeta}^i))_{n},\quad \forall 1\le i \le N,
\]
for $\zeta^1, \ldots, \zeta^i \in \Ima \Pi^c_n$, where ${ \boldsymbol \zeta}^i=(\zeta_m^i)_{m\in \mathbb Z}\in Y^c$ is defined by $\zeta_n^i:=\zeta^i$ and $\zeta_m^i=0$ for $m\neq n$. Thus, $\phi_n(0)=0$, $d_0\phi_n=0$ and $\|d_{x_c}^i\phi_n\|\le \tilde M$ near $x_c=0$ for $i=2,\ldots ,N$.

For the invariance property,
we first show that $F_n(\Gamma_n) \subset\Gamma_{n+1}$ for each $n\in\mathbb Z$. Fix $n\in\mathbb Z$ and choose any $x_c+\phi_n(x_c)\in \Gamma_n$. According to the definitions of $\phi_n$ and $\mathbb F$,
\begin{align}\label{Fv}
F_n (x_c+\phi_n(x_c))= F_n(({\bf v}+\Phi ({\bf v}))_n) = (\mathbb F({\bf v}+\Phi ({\bf v})))_{n+1},
\end{align}
where ${\bf v}\in Y^c$ is defined by~\eqref{def-vn}. Since ${\bf v}+\Phi ({\bf v}) \in \Gamma$ and $\mathbb F(\Gamma)=\Gamma$,
 there exists  $\tilde {\bf v}\in Y^c$ such that $\mathbb F({\bf v}+\Phi ({\bf v})) = \tilde {\bf v} +\Phi(\tilde {\bf v})\in \Gamma$. Set $\tilde x_c:= (\tilde{\bf v})_{n+1}\in \Ima \Pi^c_{n+1}$. We claim that
\begin{align}\label{vvxx}
(\tilde {\bf v} +\Phi(\tilde {\bf v}))_{n+1}= \tilde x_c+\phi_{n+1}(\tilde x_c).
\end{align}
Indeed, let $\hat{\bf v}=(\hat v_m)_{m\in\mathbb{Z}}\in Y^c$ be defined by
$$
\hat v_{n+1}:=\tilde x_c \quad {\rm and} \quad \hat v_m:=0 \quad {\rm for}~ m\neq n+1,
$$
which means that $(\tilde {\bf v})_{n+1}=(\hat {\bf v})_{n+1}$.
Due to the definition of $\phi_{n+1}$, it suffices to prove that
\begin{align}\label{vvvv}
\Phi(\tilde {\bf v})=\Phi(\hat {\bf v}).
\end{align}
For this purpose, we need the following result.
\begin{lm}\label{pp}
Let ${\bf z}^1=(z_n^1)_{n\in \mathbb Z}$ and ${\bf z}^2=(z_n^2)_{n\in \mathbb Z}$ belong to $Y^c$ and satisfy $z_k^1=z_k^2$ for some $k\in \mathbb Z$. Then, $(\Phi({\bf z}^1))_k=(\Phi({\bf z}^2))_k$.
\end{lm}

\begin{proof}[Proof of Lemma {\rm \ref{pp}}]
Assume that $\Phi({\bf z}^1)_k\neq \Phi({\bf z}^2)_k$. Let
${\bf y} =(y_n)_{n\in\mathbb Z}\in Y$ be defined by
\begin{equation}\label{Y}
y_n=\begin{cases}
(\Phi({\bf z}^1))_n & \text{if} ~~n\ne k,
\\
(\Phi({\bf z}^2))_k & \text{if} ~~n= k.
\end{cases}
\end{equation}
Then, we see from \eqref{F} that
\begin{equation*}
(\mathbb F^m ({\bf z}^1+{\bf y}) )_n =
\begin{cases}
F_{n-1}\circ \ldots \circ F_{n-m}(z^1_{n-m}+ (\Phi({\bf z}^1))_{n-m} ) & \text{if}~~ n\neq m+k,
\\
F_{n-1}\circ \ldots \circ F_{n-m}(z^2_{n-m}+ (\Phi({\bf z}^2))_{n-m} ) & \text{if}~~ n=m+k,
\end{cases}
\end{equation*}
and that
\begin{equation*}
(\mathbb F^{-m} ({\bf z}^1+{\bf y}) )_n =
\begin{cases}
F_{n}^{-1}\circ \ldots \circ F_{n+m-1}^{-1}(z^1_{n+m}+ (\Phi({\bf z}^1))_{n+m} ) & \text{if}~~ n+m\ne k,
\\
F_{n}^{-1}\circ \ldots \circ F_{n+m-1}^{-1}(z^2_{n+m}+ (\Phi({\bf z}^2))_{n+m} ) & \text{if}~~ n+m=k,
\end{cases}
\end{equation*}
for every $n\in \mathbb Z$ and $m\in\mathbb{N}$. This together with \eqref{cm3} easily implies that
\[
\sup_{m\ge 0}\{\gamma_2^{-m}\|\mathbb F^m ({\bf z}^1+{\bf y})\|\}<\infty,\quad\sup_{m\ge 0}\{\gamma_1^m\|\mathbb F^{-m} ({\bf z}^1+{\bf y})\|\}<+\infty.
\]
By~\eqref{cm3} we have that ${\bf z}^1+{\bf y} \in \Gamma$, and thus
 ${\bf y}=\Phi({\bf z}^1)$.  This contradicts~\eqref{Y} and the proof of lemma \ref{pp} is completed.
\end{proof}

By Lemma \ref{pp}, we have that ~\eqref{vvvv} holds,  which implies~\eqref{vvxx}. Then, combining \eqref{Fv} and~\eqref{vvxx}, we obtain that
$$
F_n (x_c+\phi_n(x_c))= (\mathbb F({\bf v}+\Phi ({\bf v})))_{n+1}
=(\tilde {\bf v} +\Phi(\tilde {\bf v}))_{n+1}=
\tilde x_c+\phi_{n+1}(\tilde x_c)\in \Gamma_{n+1},
$$
and therefore $F_n(\Gamma_n)\subset \Gamma_{n+1}$. Conversely, similar arguments imply that $\Gamma_{n+1} \subset F_n(\Gamma_n)$ and therefore $\Gamma_{n+1} =F_n(\Gamma_n)$.

It remains to observe that~\eqref{cm1} and~\eqref{cm2} follow readily from~\eqref{cm3}. The proof of proposition \ref{center} is completed.
\end{proof}

\section{Nonautonomous version of Takens theorem}
\setcounter{equation}{0}

\noindent In this section, we prove our main theorem, as described in the introduction.


\subsection{Nonautonomous Belitskii-Samavol Theorem}
We consider two sequences $f_n\colon \mathbb R^d\to \mathbb R^d$ and $g_n\colon \mathbb R^d \to \mathbb R^d$, $n\in \mathbb Z$ of $C^{N+1}$
maps such that

\begin{itemize}
\item for $n\in \mathbb Z$,
\begin{equation}\label{zeroo}
f_n(0)=g_n(0)=0 \quad \text{and} \quad d_0f_n=d_0g_n=0;
\end{equation}


\item for $2\le i\le N+1$ and $x\in \mathbb R^d$
\begin{equation}\label{deriv}
\|d_xf_n\|\le \delta,\quad\|d^i_xf_n\| \le M,
\quad\|d_xg_n\|\le \delta,\quad\|d^i_xg_n\| \le M,
\end{equation}
where $\delta, M>0$ are constants with $\delta$ being sufficiently small.
\end{itemize}
By Proposition~\ref{center}, there are sequences of nonautonomous center manifolds $\Gamma_n$ and $\Gamma'_n$, $n\in \mathbb Z$ that correspond to sequences $F_n:=A_n+f_n$ and $G_n:=A_n+g_n$ for $n\in \mathbb Z$, respectively. Similarly to~\eqref{F}, we define a map $\mathbb G\colon Y \to Y$ by
\begin{equation}\label{G}
(\mathbb G (\mathbf x))_n:=G_{n-1}(x_{n-1}):=A_{n-1}x_{n-1}+g_{n-1}(x_{n-1}),
\quad\forall n\in \mathbb Z,
\end{equation}
for $\mathbf x=(x_n)_{n\in \mathbb Z}\in Y$. Let $\Gamma$ and $\Gamma'$ be center manifolds of $\mathbb F$ and $\mathbb G$, respectively. Then, we have the following lemma.

\begin{lm}\label{der0}
Let $(F_n)_{n\in\mathbb{Z}}$ and $(G_n)_{n\in\mathbb{Z}}$ be given above and suppose that $\Gamma_n=\Gamma_n'$ for $n\in \mathbb Z$. In addition, assume that $d^i_{x_c} (F_n-G_n)=0$ for $1\le i \le N$ and $x_c\in \Gamma_n$, $n\in \mathbb Z$. Then, $\Gamma=\Gamma'$ and $d^i_{{\bf x}_c}(\mathbb F-\mathbb G)={\bf 0}$ for $1\le i \le N$ and ${\bf x}_c\in \Gamma$.
\end{lm}

\begin{proof}
We recall that
\begin{align*}
&\Gamma_n=\{x_c+\phi_n(x_c): \ x_c\in \Ima \Pi_n^c \}, \qquad \Gamma'_n=\{x_c+\phi'_n(x_c): \ x_c\in \Ima \Pi_n^c \}
\end{align*}
and
\begin{align*}
&\Gamma=\{{\bf x}_c+\Phi({\bf x}_c): \ {\bf x}_c\in Y^c \}, \qquad \Gamma'=\{{\bf x}_c+\Phi'({\bf x}_c): \ {\bf x}_c\in Y^c \},
\end{align*}
with $C^N$ maps $\Phi, \Phi' \colon Y^c\to Y^s\oplus Y^u$.
Choose  an arbitrary $\mathbf x_c=(x^c_n)_{n\in \mathbb Z}\in Y^c$.
For any $n\in\mathbb{Z}$, we denote $x^c_n=:x_c\in \Ima \Pi_n^c$. Then,  we have (see~\eqref{phi-phi}) that
$$
\phi_n(x_c)=(\Phi({\bf v}) )_n,\qquad \phi'_n(x_c)=(\Phi'({\bf v}) )_n,
$$
 where ${\bf v}\in Y^c$ is given by \eqref{def-vn}. Then, Lemma \ref{pp} tells us that
$$
(\Phi({\bf x}_c) )_n=(\Phi({\bf v}) )_n=\phi_n(x_c)=\phi'_n(x_c)=(\Phi'({\bf v}) )_n=(\Phi'({\bf x}_c) )_n,
$$
where the equality $\phi_n(x_c)=\phi'_n(x_c)$ follows from
$\Gamma_n=\Gamma_n'$. Therefore,  $\Gamma=\Gamma'$.

Next, we establish the second assertion of the lemma. By~\eqref{high-deri}, we have that
\[
\begin{split}
&d_{{\bf x}_c}^i\mathbb{F}(\boldsymbol{\xi}^1,\ldots, \boldsymbol{\xi}^i)-d_{{\bf x}_c}^i\mathbb{G}(\boldsymbol{\xi}^1,\ldots, \boldsymbol{\xi}^i) \\
& = \Big(d^i_{x^c_{n-1}} f_{n-1} (\xi^1_{n-1},\ldots,\xi^i_{n-1})-d^i_{x^c_{n-1}} g_{n-1} (\xi^1_{n-1},\ldots,\xi^i_{n-1})\Big)_{n\in\mathbb{Z}} \\
&=\Big(d_{x^c_{n-1}}^i (F_{n-1}-G_{n-1})(\xi^1_{n-1},\ldots,\xi^i_{n-1})\Big)_{n\in \mathbb Z}
=(0)_{n\in \mathbb Z}=\boldsymbol 0,
\end{split}
\]
for $1\le i \le N$, $\boldsymbol{\xi}^1, \ldots, \boldsymbol{\xi}^i\in Y^c$ and ${\bf x}_c\in \Gamma=\Gamma'$. The proof of the lemma is completed.
\end{proof}

We have the following theorem, which can be regarded as the nonautonomous version of Belitskii-Samavol Theorem.

\begin{thm}\label{BS-thm-nonauto}
Take an arbitrary $k\in \mathbb N$. Then, there exists $N_0\in \mathbb N$ with the property that for each $N\ge N_0$ such that~\eqref{gap-cond} holds and for sequences $(f_n)_{n\in \mathbb Z}$, $(g_n)_{n\in \mathbb Z}$  of $C^{N+1}$ maps on $\mathbb R^d$ satisfying \eqref{zeroo}, \eqref{deriv} and
\begin{itemize}
\item $\Gamma_n=\Gamma_n'$ and $d_{x_c}^i(f_n-g_n)=0$ for all $1\le i \le N$, $x_c\in \Gamma_n$ and $n\in \mathbb Z$,
\end{itemize}
  there exists a sequence $(H_n)_{n\in \mathbb Z}$ of $C^k$ local
 diffeomorphisms $H_n
 $ near $0$ such that the following holds:
\begin{itemize}
\item for $n\in \mathbb Z$,
\begin{equation}\label{Deriv}
H_n(0)=0;
\end{equation}
\item  there exists a constant $\rho>0$ such that for $n\in \mathbb Z$,
\begin{equation}\label{hgf}
H_{n+1}\circ F_n (x)=G_n\circ H_n(x)\qquad \text{for} \quad \|x\|\le \rho,
\end{equation}
where $F_n=A_n+f_n$ and $G_n=A_n+g_n$;
\item there exist $M, \tilde \rho>0$ such that
\begin{equation}\label{1330}
\|d_{x}^iH_n \| \le M \quad \text{and} \quad \|d_{x}^iH_n^{-1}\| \le M,
\end{equation}
for each $n\in\mathbb Z$, each $1\le i\le k$, and $x\in \mathbb R^d$
with $\|x\| \le \tilde \rho$.
\end{itemize}

\end{thm}

\begin{proof}
It follows from Lemmas \ref{lemm}, \ref{der0} and Theorem~\ref{ck-conjuga} (which can be applied as $Y$ is a Hilbert space)
that there exists a local $C^k$ diffeomorphism $\mathbb H\colon Y \to Y$ and a constant $\rho>0$ such that \begin{equation}\label{2022}
\mathbb H(\boldsymbol 0)=\boldsymbol 0,
\end{equation}
and
\begin{equation}\label{GHF}
\mathbb  H\circ \mathbb  F (\mathbf{x})=\mathbb G\circ \mathbb H  (\mathbf{x}) \qquad \text{for} \qquad  \|\mathbf{x}\|\le \rho,
\end{equation}
where $\mathbb F$ and $\mathbb G$ are given by~\eqref{F} and~\eqref{G}, respectively.
Similarly to \eqref{phi-phi}, the desired local conjugacy $H_n
$ can be constructed by
$$
H_n(x):=(\mathbb H({\bf w}))_n, \quad  \text{for  $\|x\| \le \rho$ and
$ n\in \mathbb{Z}$},
$$
where ${\bf w}:=(w_m)_{m\in\mathbb{Z}}\in Y$ is defined by
$w_n:=x$ and $w_m:=0$ for $m\neq n$. Then, similarly to \eqref{F} and \eqref{high-deri}, we see that $H_n$ is $C^k$ and that
\begin{equation}\label{ttt}
d_{x}^i H_n(\zeta^1, \ldots, \zeta^i)=(d_{\bf w}^i \mathbb H(\boldsymbol{\zeta}^1, \ldots, \boldsymbol{\zeta}^i))_n,
\end{equation}
for $1\le i \le k$ and $\zeta^1, \ldots ,\zeta^i \in \mathbb R^d$, where $\boldsymbol{\zeta}^j=(\zeta_m^j)_{m\in \mathbb Z}\in Y$ is defined by $\zeta_n^j:=\zeta^j$ and $\zeta_m^j:=0$ for $m\neq n$. Moreover, each $H_n$ is a local diffeomorphism and $H_n^{-1}(x)=(\mathbb H^{-1}(\mathbf w))_n$.

Observe that~\eqref{Deriv} follows readily from~\eqref{2022} and~\eqref{ttt} with ${\bf w}={\bf 0}$ and $i=1$ and that \eqref{hgf} follows directly from~\eqref{GHF}. On the other hand, since $\mathbb H$ and $\mathbb H^{-1}$ are $C^k$, there exist $M, \tilde\rho>0$ such that for $1\le i\le k$,
\[
\|d_{{\bf w}}^i\mathbb H  \| \le M \quad \text{and} \quad  \|d_{{\bf w}}^i\mathbb H^{-1}\| \le M, \quad \text{for ${\bf w}\in Y$ with $\|{\bf w}\| \le \tilde \rho$.}
\]
Hence, by \eqref{ttt} we obtain that
\[
\begin{split}
\|d_{x}^iH_n (\xi^1, \ldots, \xi^i)\|
&=\|(d_{{\bf w}}^i \mathbb H(\boldsymbol{\zeta}^1, \ldots, \boldsymbol{\zeta}^i))_n\|\le \|d_{\bf w}^i \mathbb H(\boldsymbol{\zeta}^1, \ldots, \boldsymbol{\zeta}^i)\|
\\
& \le M\prod_{j=1}^i \|\boldsymbol{\zeta}^j \|
= M\prod_{j=1}^i  \|\zeta^j\|,
\end{split}
\]
for each $n\in\mathbb Z$ and  $x\in \mathbb R^d$ such that $\|x\| \le \tilde \rho$ (so that $\|{\bf w}\| \le \tilde \rho$) and $1\le i \le k$. This establishes the first estimate in~\eqref{1330}. Similarly, one can establish the second one and the proof of the theorem is completed.
\end{proof}

\subsection{Nonautonomous Takens Theorem}
For $q:=(q_1,\ldots ,q_r),~t:=(t_1,\ldots ,t_r)\in \mathbb{N}_0^r$ with $\mathbb{N}_0:=\mathbb{N}\cup \{0\}$, $a:=(a_1, \ldots ,a_r)\in \mathbb{R}^r$, $b:=(b_1,\ldots ,b_r)\in \mathbb{R}^r$, we denote
\begin{align*}
&|q|:=q_1+\ldots+q_r,\quad [a,b]^q:=[a_1^{q_1}\cdots a_r^{q_r},~b_1^{q_1}\cdots b_r^{q_r}],
\\
&|t|:=t_1+\ldots+t_r,\quad [a,b]^t:=[a_1^{t_1}\cdots a_r^{t_r},~b_1^{t_1}\cdots b_r^{t_r}],
\end{align*}
where $a_i$ and $b_i$ for $1\le i \le r$ are as in~\eqref{SA}.
In addition, for $N\in\mathbb{N}$ and intervals $[c_1,c_2]$, $[c_3,c_4]$ with $c_i\in\mathbb{R}^+$ for $i\in \{1, 2, 3, 4\}$, we denote
$$
[c_1,c_2]^N:=[c_1^N,c_2^N], \quad[c_1,c_2]\cdot[c_3,c_4]:=[c_1c_3,~c_2c_4].
$$
Then, based on Theorem \ref{BS-thm-nonauto}, we reformulate and prove the main theorem of this paper as follows.
\begin{thm}\label{Thm-Takens}
For any $k\in\mathbb{N}$, let $N:=3N_0+1$ where $N_0\in \mathbb{N}$ {\rm (}depending on $k${\rm )} is given in Theorem {\rm \ref{BS-thm-nonauto}}. Suppose that the non-resonance conditions
\begin{align}
&\left([a,b]^q\cdot[\nu_-,\nu_+]^{N} \right)\cap [\nu_-,\nu_+]=\emptyset,
\label{NR-1}\\
&\left( [a,b]^t\cdot[\nu_-,\nu_+]^{N} \right)\cap [a_i,b_i]=\emptyset,\quad i=1,\ldots ,r,
\label{NR-2}
\end{align}
hold for all $|q|\in \{1,\ldots ,N\}$ and $|t|\in \{2, \ldots ,N\}$. Then, any sequence $(F_n)_{n\in\mathbb{Z}}$ of $C^{N+1}$
maps of the form $F_n:=A_n+f_n$ satisfying~\eqref{zero} and~\eqref{small-Lip} is locally $C^k$ conjugated to the sequence of maps
\begin{align}\label{Takens-NF}
\left(
  \begin{array}{lll}
x_s
\vspace{1ex}\\
x_c
\vspace{1ex}\\
x_u
\end{array}
\right)
\mapsto
\left(
  \begin{array}{lll}
A_n^s(x_c)x_s
\vspace{1ex}\\
A_n^cx_c+f_n^c(x_c)
\vspace{1ex}\\
A_n^u(x_c)x_u
\end{array}
\right),\quad \forall n\in\mathbb{Z},
\end{align}
by a conjugacy $(\Psi_n)_{n\in\mathbb{Z}}$, where $A_n^s(x_c): X_s\to X_s$, $A_n^u(x_c): X_u\to X_u$ are linear operators which
are $C^{N-1}$ in $x_c$ such that $A_n^s(0)=A_n^s$, $A_n^u(0)=A_n^u$, and $f_n^c:X_c\to X_c$ is a $C^N$ map. Moreover, for each $n\in\mathbb Z$,
\begin{align}\label{bounded-psi}
\Psi_n^{\pm}(0)=0,
\qquad \|d_{x}^i\Psi_n^{\pm}\|\le \tilde M,\quad i=1,\ldots ,k,
\quad {\rm for}~ x\in\mathbb{R}^d~{\rm near}~0,
\end{align}
where $ \tilde M>0$ is a constant.
Here, $\Psi_n^+=\Psi_n$ and $\Psi_n^-=\Psi_n^{-1}$.
\end{thm}

Remark that \eqref{NR-1} implies
\begin{align}\label{NR-cond-1}
(a_1^{q_1}\cdots a_r^{q_r})\nu_-^{N}>\nu_+\quad {\rm or}\quad (b_1^{q_1}\cdots b_r^{q_r})\nu_+^{N}<\nu_-
\end{align}
and \eqref{NR-2} implies
\begin{align}\label{NR-cond-2}
(a_1^{t_1}\cdots a_r^{t_r})\nu_-^{N}>b_i\quad {\rm or}\quad (b_1^{t_1}\cdots b_r^{t_r})\nu_+^{N}<a_i,\quad i=1,\ldots ,r,
\end{align}
by the notations introduced before the statement of the theorem.

\noindent {\it Proof of Theorem} \ref{Thm-Takens}.
By Proposition \ref{center}, the sequence $(F_n)_{n\in\mathbb Z}$ admits  a nonautonomous $C^N$
center manifold,
which is the sequence of graphs of $C^N$ maps
$
\phi_n:X_c\to X_s\oplus X_u.
$
This enables us to define a sequence of $C^N$ transformations $\varphi_n:\mathbb R^d \to \mathbb R^d$, $n\in\mathbb Z$ by
\begin{equation}\label{trans-center}
    \begin{cases}
        \tilde x_c := x_c,\\
        (\tilde x_s,\tilde x_u):=(x_s,x_u)-\phi_n(x_c),
    \end{cases}
\end{equation}
which satisfies that
\begin{align}\label{bounded-0}
\begin{split}
&\varphi_n^{\pm}(0)=0, \quad d_0\varphi_n^{\pm}=\Id,\quad \|d_{x}\varphi_n^{\pm}(x)-\Id\|\le \tilde \delta,
\\
&\|d_{x}^i\varphi_n^{\pm}(x)\|\le \tilde M,\quad i=2,\ldots ,N,
\quad {\rm for}~ x~{\rm near}~0,
\end{split}
\end{align}
where $\varphi_n^+=\varphi_n$, $\varphi_n^-=\varphi_n^{-1}$ and $\tilde \delta, \tilde M>0$ with the possibility of choosing $\tilde \delta$ arbitrarily small.
Using $(\varphi_n)_{n\in\mathbb Z}$ one can conjugate
$(F_n)_{n\in\mathbb Z}$ to the sequence  $(\tilde F_n)_{n\in\mathbb Z}$, i.e,
$$
\tilde F_n:=\varphi_{n+1}\circ F_n\circ \varphi_n^{-1}=A_n+\tilde f_n,
$$
whose center manifold is straightened up, i.e., $X_c$ is the (nonautonomous) center manifold of $(\tilde F_n)_{n\in\mathbb Z}$.
Notice that, by \eqref{small-Lip}, \eqref{bounded-0} and using smooth cut-off functions, we can assume that
the following conditions hold:
\begin{itemize}
\item $\tilde f_n(0)=0$ and $d_0\tilde f_n=0$;
\item for $2\le i\le  N$,
\begin{equation}\label{small-Lip-tt}
\|d_x\tilde f_n\|\le \tilde \delta,\quad\|d^i_x\tilde f_n\| \le \tilde M, \quad \forall x\in \mathbb R^d.
\end{equation}
\end{itemize}
Put $v:=x_s+x_u\in X_{su}$ and write $\tilde F_n(x_s,x_c,x_u)$ as $\tilde F_n(x_c,v)$. For an integer $p\in\{1,\ldots,N_0\}$,
assume that
$$
\partial_v^i (\pi_c\tilde F_n) (x_c,0)=0\quad{\rm for}~ i=1,\ldots,p-1.
$$
Then, the Taylor expansion of $\tilde F_n$ at $(x_c,0)$ is given by
\begin{equation}\label{after-straight}
    \tilde F_n (x_c,v)=\Big(w_n(x_c)+f_n^p(x_c,v)
    +o(\|v\|^p),~A_n^{su}(x_c)v+O(\|v\|^2)\Big)
\end{equation}
where $w_n(x_c):=\pi_c\tilde F_n(x_c,0)$ is $C^{N-p}$ (actually $C^{N}$) and $A_n^{su}(x_c):=\partial_v (\pi_{su}\tilde F_n) (x_c,0)$ is $C^{N-p}$ (actually $C^{N-1}$)
such that
$$
Dw_n(0)=A_n^c,\qquad A_n^{su}(0)={\rm diag}(A_n^s,A_n^u).
$$
Moreover, $f_n^p(x_c,v):=\partial_v^p (\pi_c\tilde F_n) (x_c,0)v^p$ is a homogeneous polynomial in variable $v$ of order $p$ with coefficients being $C^{N-p}$ maps in $x_c$. In order to eliminate the term $f_n^p(x_c,v)$, we need the following lemma. We note that
below and throughout the rest of the proof, $\tilde \delta, \tilde M>0$ will denote generic constants which can change from one occurrence to the next one and $\tilde \delta$ can be made arbitrarily small.
\begin{lm}\label{rem-cent}
The sequence $(\tilde F_n)_{n\in\mathbb{Z}}$ of $C^{N-p}$ maps is locally conjugated to a sequence $(\hat F_n)_{n\in\mathbb{Z}}$ of $C^{N-p-1}$ maps defined by
$$
\hat F_n(x_c,v):=\Big(w_n(x_c)+o(\|v\|^p),~A_n^{su}(x_c)v+O(\|v\|^2)\Big)
$$
via a sequence $(H_n^p(x_c,v))_{n\in\mathbb{Z}}$, where each $H_n^p-\Id:\mathbb{R}^d\to \mathbb{R}^d$ is a homogeneous polynomial in variable $v$ of order $p$ with coefficients being $C^{N-p-1}$ maps in $x_c$. Moreover,
\begin{align}\label{bounded}
\begin{split}
&(H_n^p)^{\pm}(0)=0,\quad d_0(H_n^p)^{\pm}=\Id, \quad \|d_{x}(H_n^p)^{\pm}-\Id\|\le \tilde \delta,
\\
&\|d_{x}^i(H_n^p)^{\pm}\|\le \tilde M,\quad i=2,\ldots ,N-p,
\quad {\rm for}~ x=(x_c,v)~{\rm near}~0.
\end{split}
\end{align}
\end{lm}

\noindent {\it Proof of Lemma} \ref{rem-cent}.
Let $h_n^p:X_c\times X_{su}\to X_c$, $n\in\mathbb Z$ be a sequence of homogeneous polynomials in variable $v$ of order $p$ with coefficients being of  a certain smoothness in $x_c$ (which will be proved to be $C^{N-p}$),
and consider the transformations
\begin{equation}\label{elimi-coor-chang}
    H_n^p:(x_c,v)\mapsto (x_c+h_n^p(x_c,v),v)
\end{equation}
together with their inverses
\begin{equation}\label{elimi-coor-chang-inver}
    (H_n^p)^{-1}:(x_c,v)\mapsto (x_c-h_n^p(x_c,v)+o(\|v\|^p),v).
\end{equation}
Combining \eqref{after-straight} with \eqref{elimi-coor-chang}, we calculate that
\begin{align*}
  &H^p_{n+1}\circ \tilde F_n(x_c,v)
   \\
   & = \Big(w_n(x_c) + f_n^p(x_c,v)+ h_{n+1}^p(w_n(x_c),A_n^{su}(x_c)v)+o(\|v\|^p),
   A_n^{su}( x_c)v+O(\|v\|^2)\Big).
\end{align*}
Employing \eqref{elimi-coor-chang-inver}, we further obtain that
\begin{align*}
&   H^p_{n+1}\circ \tilde F_n\circ (H_n^p)^{-1}(x_c,v)
   \\
   & = \Big(w_n(x_c) - Dw_n(x_c) h_n^p(x_c,v) + f_n^p(x_c,v)+ h_{n+1}^p(w_n(x_c),A_n^{su}(x_c)v)+o(\|v\|^p),
   \\
   &\qquad A_n^{su}(x_c)v+O(\|v\|^2) \Big),
\end{align*}
where the term $O(\|v\|^2)$ has the same smoothness as $h_n^p(\cdot,v)$.
In order to eliminate the term of order $p$, it suffices to solve the following nonautonomous equation
\begin{equation}\label{elimi-equa-ori}
    - Dw_n(x_c) h_n^p(x_c,v) + f_n^p(x_c,v)+ h_{n+1}^p(w_n(x_c),A_n^{su}(x_c)v)=0, \quad \forall n\in\mathbb Z,
\end{equation}
which is equivalent to
\begin{equation}\label{elimi-equa-chan}
    - Dw_n(x_c) h_n^p(x_c,A_n^{su}(x_c)^{-1}v) + f_n^p(x_c,A_n^{su}(x_c)^{-1}v)+ h_{n+1}^p(w_n(x_c),v)=0,
\end{equation}
by replacing $v$ with $A_n^{su}(x_c)^{-1}v$. We recall that all the known maps in the above equations are $C^{N-p}$ in $x_c$. Equation \eqref{elimi-equa-chan} can be seen as  a {\it nonautonomous homological equation} (cf. \cite[p.182]{Arnold-book}), which depends on the center variable $x_c$. This is the essential difference from the hyperbolic case and we will use the center manifold theory to solve it.

Let $\Xi^p$
be the linear space of all homogeneous vector polynomials of order $p$ in the variable $v$ whose elements are of the form
$$
\sum_{q}C^qx_1^{q_1}\cdots x_r^{q_r},
$$
where $C^q:X_1^{q_1}\times\cdots \times X_r^{q_r}\to X_c$ are multilinear operators and $|q|=p$. Let $\Upsilon_n: \tilde \Xi^p_c:=X_c\times \Xi^p\to \tilde \Xi^p_c$, $n\in\mathbb Z$ be a sequence of $C^{N-p}$ maps defined by
\begin{equation}\label{linear-susp}
  \Upsilon_n
 \colon (x_c,\psi^p(v))\mapsto \Big(w_n(x_c), ~L_n(x_c) \psi^p(v) -f_n^p(x_c,A_n^{su}(x_c)^{-1}v) \Big),
\end{equation}
where $L_n(x_c):\Xi^p\to \Xi^p$ is a linear operator defined by
\begin{equation}\label{def-L}
    L_n(x_c) \psi^p(v) := Dw_n(x_c) \psi^p(A_n^{su}(x_c)^{-1}v),\quad \forall \psi^p(v)\in \Xi^p.
\end{equation}
We observe that if the sets ${\mathcal M}_n:=\{(x_c,h_n^p(x_c,v))\in \tilde \Xi^p_c: x_c\in X_c\}$, $n\in\mathbb{Z}$ are invariant under $(\Upsilon_n)_{n\in\mathbb{Z}}$, i.e., $\Upsilon_n({\mathcal M}_n)\subset{\mathcal M}_{n+1}$,
then
$$
h_{n+1}^p(w_n(x_c),v)=L_n(x_c) h_n^p(x_c,v) -f_n^p(x_c,A_n^{su}(x_c)^{-1}v),
$$
which is exactly the equation \eqref{elimi-equa-chan}. Thus, we convert the problem of solving equation \eqref{elimi-equa-chan} to the problem of finding an invariant manifold $({\mathcal M}_n)_{n\in\mathbb{Z}}$ of $(\Upsilon_n)_{n\in\mathbb{Z}}$.

In order to find an invariant manifold of $(\Upsilon_n)_{n\in\mathbb{Z}}$, we consider the shift $\tilde \Upsilon_n$ of $\Upsilon_n$ defined by
\begin{align}\label{linear-susp-1}
(x_c,\psi^p(v))
\mapsto \Big(w_n(x_c), ~L_n(x_c) (\psi^p(v)+\kappa_n^p(v))-f_n^p(x_c,A_n^{su}(x_c)^{-1}v) -\kappa_{n+1}^p(v)\Big),
\end{align}
where $(\kappa_n^p(v))_{n\in\mathbb{Z}}$ is a solution of the equation
\begin{align}\label{def-kap}
L_n(0)\kappa_n^p(v)-\kappa_{n+1}^p(v)=f_n^p(0,A_n^{su}(0)^{-1}v).
\end{align}
Set\begin{align*}
&S^p_+:=\{q:=(q_1,\ldots ,q_r)\in \mathbb{N}_0^r:
a_1^{q_1}\cdots a_r^{q_r}>\nu_+,~|q|=p\},
\\
&S^p_-:=\{q:=(q_1,\ldots ,q_r)\in \mathbb{N}_0^r:
b_1^{q_1}\cdots b_r^{q_r}<\nu_-,~|q|=p\},
\end{align*}
and let $\Xi^{p,+}, \Xi^{p,-}$ be two subspaces of $\Xi^p$ with elements
$$
\psi^{p,+}(v):=\sum_{q\in S_+^p}C_+^qx_1^{q_1}\cdots x_r^{q_r},\qquad
\psi^{p,-}(v):=\sum_{q\in S_-^p}C_-^qx_1^{q_1}\cdots x_r^{q_r},
$$
respectively, where $C_\pm^q:X_1^{q_1}\times\cdots \times X_r^{q_r}\to X_c$ are multilinear operators. By \eqref{NR-cond-1} we know that either $a_1^{q_1}\cdots a_r^{q_r}\ge a_1^{q_1}\cdots a_r^{q_r}\nu_-^N>\nu_+$ or $b_1^{q_1}\cdots b_r^{q_r}\le b_1^{q_1}\cdots b_r^{q_r}\nu_+^N<\nu_-$, which implies that
\begin{align}\label{xi+-}
\Xi^p=\Xi^{p,+}\oplus\Xi^{p,-},
\end{align}
and by \eqref{def-L} we know that $\Xi^{p,\pm}$ are invariant under $(L_n(0))_{n\in\mathbb{Z}}$. Let
${\mathcal L}(m,n)$ be the linear cocycle generated by $(L_n(0))_{n\in\mathbb{Z}}$.
By iterating \eqref{def-kap}, we obtain that
\begin{align}\label{posi-ite}
\kappa_{n}^p(v) = {\mathcal L}(n,m) \kappa_{m}^p(v)-\sum_{i=m}^{n-1} {\mathcal L}(n,i+1) f_i^p(0,A_i^{su}(0)^{-1}v)
\end{align}
for $n\ge m$ and that
\begin{align}\label{neg-ite}
\kappa_{n}^p(v) = {\mathcal L}(n,m) \kappa_{m}^p(v) + \sum_{i=n}^{m-1} {\mathcal L}(n,i+1) f_i^p(0,A_i^{su}(0)^{-1}v)
\end{align}
for $n\le m$.

In order to further compute the above $ \kappa_{m}^p(v)$, we define the norm
$$
\|\psi^{p,\pm}(v)\|:=\max_{q\in S_\pm^p}\sup_{\|x_1\|\ne 0,\ldots ,\|x_r\|\ne 0}\frac{\|C_\pm^qx_1^{q_1}\cdots x_r^{q_r}\|}{\|x_1\|^{q_1}\cdots \|x_r\|^{q_r}},\qquad \forall \psi^{p,\pm}(v)\in \Xi^{p,\pm},
$$
and let
$$
{\mathcal A}_c(n,m):={\mathcal A}(n,m)|_{X_c},\qquad
{\mathcal A}_i(n,m):={\mathcal A}(n,m)|_{X_i},\quad i=1,\ldots ,r.
$$
Then, for $n\ge m$ we obtain from \eqref{def-L} and \eqref{new-ET} (see also Remark~\ref{simpl}) that
\begin{align}\label{L+}
& \|{\mathcal L}(n,m)\psi^{p,+}(v)\|
\nonumber\\
&=\|Dw_{n-1}(0)\cdots Dw_m(0) \psi^{p,+}(A_m^{su}(0)^{-1}\cdots A_{n-1}^{su}(0)^{-1}v)\|
\nonumber\\
&= \max_{q\in S_+^p}\sup_{\|x_1\|\ne 0,\ldots ,\|x_r\|\ne 0}
\frac{\|{\mathcal A}_c(n,m) C_+^q\{{\mathcal A}_1(m,n)x_1\}^{q_1}\cdots\{{\mathcal A}_r(m,n)x_r\}^{q_r}\|}{\|x_1\|^{q_1}\cdots \|x_r\|^{q_r}}
\nonumber\\
&\le C \nu_+^{n-m}\max_{q\in S_+^p}\sup_{\|x_1\|\ne 0,\ldots,\|x_r\|\ne 0}\frac{\|C_+^q\{{\mathcal A}_1(m,n)x_1\}^{q_1}\cdots\{{\mathcal A}_r(m,n)x_r\}^{q_r}\|}{\|{\mathcal A}_1(m,n)x_1\|^{q_1}\cdots\|{\mathcal A}_r(m,n)x_r\|^{q_r}}
\nonumber\\
&\quad\cdot\frac{\|{\mathcal A}_1(m,n)x_1\|^{q_1}\cdots\|{\mathcal A}_r(m,n)x_r\|^{q_r}}{\|x_1\|^{q_1}\cdots \|x_r\|^{q_r}}
\nonumber\\
&\le C\nu_+^{n-m}\max_{q\in S_+^p}(a_1^{q_1}\cdots a_r^{q_r})^{-(n-m)}\|\psi^{p,+}(v)\|
\nonumber\\
&= C\mu_1^{n-m}\|\psi^{p,+}(v)\|,\qquad \forall \psi^{p,+}(v)\in \Xi^{p,+},
\end{align}
where $\mu_1:=\max_{q\in S_+^p}\{\nu_+/(a_1^{q_1}\cdots a_r^{q_r})\}\in (0,1)$. Similarly, for $n\le m$ we also have that
\begin{align}\label{L-}
\|{\mathcal L}(n,m)\psi^{p,-}(v)\|
\le C\mu_2^{m-n}\|\psi^{p,-}(v)\|,\qquad \forall \psi^{p,-}(v)\in \Xi^{p,-},
\end{align}
where $\mu_2:=\max_{q\in S_-^p}\{(b_1^{q_1}\cdots b_r^{q_r})/\nu_-\}\in (0,1)$.  Now,
projecting \eqref{posi-ite} onto $\Xi^{p,+}$ (letting $m\to -\infty$) and \eqref{neg-ite} onto $\Xi^{p,-}$ (letting $m\to +\infty$), respectively, we see from \eqref{L+}-\eqref{L-} that
\begin{align*}
 \kappa_{n}^p(v)
 &= -\sum_{i=-\infty}^{n-1} {\mathcal L}(n,i+1) \Pi^{+}  f_i^p(0,A_i^{su}(0)^{-1}v)
 \\
 & \quad +\sum_{i=n}^\infty {\mathcal L}(n,i+1) \Pi^{-}  f_i^p(0,A_i^{su}(0)^{-1}v),
\end{align*}
where $\Pi^\pm:\Xi^p\to \Xi^{p,\pm}$ are projections corresponding to the decomposition \eqref{xi+-}.
It implies that
 for any $n\in\mathbb{Z}$,
\begin{align}\label{kapa-0}
\|\kappa_{n}^p(v)\|
&\le C\left(\sum_{i=-\infty}^{n-1}\mu_1^{n-i-1}
+\sum_{i=n}^{+\infty}\mu_2^{i+1-n}\right)\|f_i^p(0,A_i^{su}(0)^{-1}v)\|
\nonumber\\
&\le C\left(\sum_{i=-\infty}^{n-1}\mu_1^{n-i-1}
+\sum_{i=n}^{+\infty}\mu_2^{i+1-n}\right)\tilde M=:M_\kappa
\end{align}
by \eqref{small-Lip-tt} since $f_n^p(x_c,v):=\partial_v^p (\pi_c\tilde F_n) (x_c,0)v^p$ (see~\eqref{after-straight}), where
we set
\[
 \| \psi^p(v)\|:=\max \{\| \Pi^+\psi^p(v)\|, \|\Pi^- \psi^p(v)\|\}, \quad \psi^p(v)\in \Xi^p.
\]
Thus, we find the solution $\kappa_{n}^p(v)$ of equation \eqref{def-kap} such that \eqref{kapa-0} holds.

Next, we continue to investigate the map $\tilde \Upsilon_n$ given by \eqref{linear-susp-1}. It is clear that $(0,0)$ is the fixed point of $\tilde \Upsilon_n$ by \eqref{def-kap}
and that the derivative of $\tilde \Upsilon_n$ at $(0,0)$ has the form
\begin{equation}\label{block-low}
    {\Delta}_n=\begin{pmatrix}
    A_n^c &0\\
    M_n(v)&L_n(0)
    \end{pmatrix},
\end{equation}
where $A_n^c=Dw_n(0)$ and $M_n(v):X_c\to \Xi^p$ is given by
\begin{align*}
&M_n(v)
\\
&:=\frac{\partial\{L_n(x_c) (\psi^p(v)+\kappa_n^p(v))-f_n^p(x_c,A_n^{su}(x_c)^{-1}v) -\kappa_{n+1}^p(v)\}}{\partial x_c}\Big|_{x_c=0,~\psi^p(v)=0}
\\
&=\frac{\partial\{Dw_n(x_c) \kappa_n^p(A_n^{su}(x_c)^{-1}v)-f_n^p(x_c,A_n^{su}(x_c)^{-1}v) \}}{\partial x_c}\Big|_{x_c=0},
\end{align*}
by using~\eqref{def-L}.
Recalling that
$w_n(x_c)=\pi_c\tilde F_n(x_c,0)$,
 and by \eqref{small-Lip-tt} and \eqref{kapa-0}, we  see that
\begin{align}\label{M-bound}
\|M_n(v)\|\le \tilde M,\qquad \forall n\in\mathbb{Z}.
\end{align}
Now, we claim that the 
sequence $(\Delta_n)_{n\in\mathbb{Z}}$
admits an exponential trichotomy  and that  its dichotomy spectrum satisfies
\begin{align}\label{DS-Delt}
\Sigma(\Delta_n)\subset(0, \mu_1]\cup[\nu_-,\nu_+]\cup[\mu_2^{-1},+\infty).
\end{align}
In fact, notice that $\tilde\Xi^{p,\pm}:=\{0\}\times\Xi^{p,\pm}$ are two invariant
subspaces of $\tilde \Xi^p_c:=X_c\times \Xi^p$ under $(\Delta_n)_{n\in\mathbb{Z}}$ and that we can equip $\tilde \Xi^p_c$ with the norm
\[
\|(x_c, \psi^p(v))\|:=\max \{\|x_c\|, \|\psi^p(v)\| \}, \quad (x_c, \psi^p(v))\in \tilde \Xi^p_c.
\]
Moreover,  \eqref{L+} and \eqref{L-} give that
\begin{align}\label{susp-ET-su-0}
\begin{split}
&\|\Delta_{m-1}\cdots \Delta_n(0,\psi^{p,+}(v))\|
\\
&=\|{\mathcal L}(m,n)\psi^{p,+}(v)\|
\le C\mu_1^{m-n}\|(0, \psi^{p,+}(v))\| \quad {\rm for}~ m\ge n,
\\
&\|\Delta_{m}^{-1}\cdots \Delta_{n-1}^{-1}(0,\psi^{p,-}(v))\|
\\
&=\|{\mathcal L}(m,n)\psi^{p,-}(v)\|\le C\mu_2^{n-m}\|(0, \psi_n^{p,-}(v))\|
 \quad {\rm for}~ m\le n.
\end{split}
\end{align}
Consequently,
$\tilde \Xi^{p,\pm}$ can be regarded as the stable and unstable subspaces of $(\Delta_n)_{n\in\mathbb{Z}}$, respectively.

Furthermore, in order to find the center subspace of $(\Delta_n)_{n\in\mathbb{Z}}$, for
any given $x_c\in X_c$, let $\epsilon_n(v)\in \Xi^{p}$, $n\in\mathbb{Z}$
be defined by
\begin{align}\label{ep-def}
\epsilon_n(v):=\sum_{i=-\infty}^{n-1}{\mathcal L}(n,i+1)\Pi^+S_i(v)x_c
-\sum_{i=n}^{+\infty}{\mathcal L}(n,i+1)\Pi^-S_{i}(v)x_c,
\end{align}
where
$
S_i(v):=M_i(v){\mathcal A}_c(i,0):X_c\to \Xi^p
$
for $i\in\mathbb{Z}$. We claim that $\epsilon_n(v)$ is well-defined.
 In fact, by \eqref{new-ET} (see also Remark~\ref{simpl}) and \eqref{M-bound}, we have
\begin{align*}
&\sup_{i\le 0}\{\nu_-^{-i-1}\|S_i(v)\|\}\le \tilde M\sup_{i\le 0}\{\nu_-^{-i-1}\|{\mathcal A}_c(i,0)\|\}<\infty,
\\
&\sup_{i\ge 0}\{\nu_+^{-i-1}\|S_i(v)\|\}\le \tilde M\sup_{i\ge 0}\{\nu_+^{-i-1}\|{\mathcal A}_c(i,0)\|\}<\infty.
\end{align*}
The above together with~\eqref{L+} and~\eqref{L-} implies that for any $n\ge 0$
\begin{align*}
\|\epsilon_n(v)\|
&\le \nu_-^{n}\sum_{i=-\infty}^{0}C(\mu_1/\nu_-)^{n-i-1}\nu_-^{-i-1}\|\Pi^+S_{i}(v)\|
\\
&\quad+ \nu_+^{n}\sum_{i=1}^{n-1}C(\mu_1/\nu_+)^{n-i-1}\nu_+^{-i-1}\|\Pi^+S_{i}(v)\|
\\
&\quad+\nu_+^{n}\sum_{i=n}^{+\infty}C(\mu_2\nu_+)^{i+1-n}\nu_+^{-i-1}\|\Pi^-S_{i}(v)\|<\infty,
\end{align*}
and that for any $n\le 0$
\begin{align*}
\|\epsilon_n(v)\|
&\le \nu_-^{n}\sum_{i=-\infty}^{n-1}C(\mu_1/\nu_-)^{n-i-1}\nu_-^{-i-1}\|\Pi^+S_{i}(v)\|
\\
&\quad+ \nu_-^{n}\sum_{i=n}^{-1}C(\mu_2\nu_-)^{i+1-n}\nu_-^{-i-1}\|\Pi^-S_{i}(v)\|
\\
&\quad+\nu_+^{n}\sum_{i=0}^{+\infty}C(\mu_2\nu_+)^{i+1-n}\nu_+^{-i-1}\|\Pi^-S_{i}(v)\|<\infty,
\end{align*}
since
\begin{align}\label{SG-cond-1}
\begin{split}
&\frac{\mu_1}{\nu_-}=\max_{q\in S_+^p}\{\nu_+/(\nu_-a_1^{q_1}\cdots a_r^{q_r})\}\in (0,1),\quad \frac{\mu_1}{\nu_+}=\max_{q\in S_+^p}\{1/a_1^{q_1}\cdots a_r^{q_r}\}\in (0,1),
\\
&\mu_2\nu_+=\max_{q\in S_-^p}\{(\nu_+b_1^{q_1}\cdots b_r^{q_r})/\nu_-\}\in (0,1),\quad \mu_2\nu_-=\max_{q\in S_-^p}\{b_1^{q_1}\cdots b_r^{q_r}\}\in (0,1),
\end{split}
\end{align}
due to \eqref{NR-cond-1}.
 Then,  $\epsilon_n(v)$ is well-defined for $n\in\mathbb{Z}$ and
\begin{align}\label{center+-}
\sup_{n\ge 0}(\nu_+^{-n}\|\epsilon_n(v)\|)<\infty,\qquad\sup_{n\le 0}(\nu_-^{-n}\|\epsilon_n(v)\|)<\infty.
\end{align}
Moreover, we verify that
\begin{align*}
\Delta_n({\mathcal A}_c(n,0)x_c, \epsilon_n(v))
&=\Big({\mathcal A}_c(n+1,0)x_c,~M_n{\mathcal A}_c(n,0)x_c+L_n(0)\epsilon_n(v)\Big)
\\
&=\Big({\mathcal A}_c(n+1,0)x_c,~\sum_{i=-\infty}^{n-1}{\mathcal L}(n+1,i+1)\Pi^+S_i(v)
\\
&\qquad-\sum_{i=n}^{+\infty}{\mathcal L}(n+1,i+1)\Pi^-S_{i}(v)+S_{n}(v)\Big)
\\
&=({\mathcal A}_c(n+1,0)x_c, \epsilon_{n+1}(v)),
\end{align*}
and therefore the sequence $\{({\mathcal A}_c(n,0)x_c, \epsilon_n(v))\in \tilde \Xi^p_c:n\in\mathbb{Z}\}$ is an orbit of $(\Delta_n)_{n\in\mathbb{Z}}$. Conversely, using the Variation of Constants Formula, we can prove that for any given $x_c\in X_c$ if $(\epsilon_n(v))_{n\in\mathbb{Z}}\subset \Xi^p$ satisfies \eqref{center+-} and $\{({\mathcal A}_c(n,0)x_c, \epsilon_n(v))\in \tilde \Xi^p_c:n\in\mathbb{Z}\}$ is an orbit of $(\Delta_n)_{n\in\mathbb{Z}}$, then $\epsilon_n(v)$ are given by \eqref{ep-def}. We conclude that
\begin{itemize}
\item
for any $x_c\in X_c$, there is a unique $(\epsilon_n(v))_{n\in\mathbb{Z}}$ such that $\{({\mathcal A}_c(n,0)x_c, \epsilon_n(v))\in \tilde \Xi^p_c:n\in\mathbb{Z}\}$ is an orbit of $(\Delta_n)_{n\in\mathbb{Z}}$.
\end{itemize}
The sequence $(\epsilon_n(v))_{n\in\mathbb{Z}}$ with the above property corresponding to $x_c\in X_c$
will be denoted by $(\epsilon_n(x_c,v))_{n\in\mathbb{Z}}$.

The above discussion enables us to define
$$
\tilde \Xi_n^c:=\{(x_c,\epsilon_n(x_c,v))\in \tilde \Xi^p_c:x_c\in X_c\},\quad \forall n\in\mathbb{Z},
$$
which is clearly a (nonautonomous) linear subspace invariant under $(\Delta_n)_{n\in\mathbb{Z}}$, i.e. $\Delta_n \tilde \Xi_n^c=\tilde \Xi_{n+1}^c$ for $n\in \mathbb Z$.
Moreover,
$$
\tilde \Xi^p_c=\tilde\Xi^{p,+}\oplus \tilde\Xi_n^c \oplus \tilde\Xi^{p,-}, \quad n\in \mathbb Z.
$$
Denote the projections corresponding to the above decomposition by
$\tilde\Pi_n^+$, $\tilde\Pi_{n}^c$ and $\tilde\Pi_n^-$, respectively. Note that
\begin{align*}
&\tilde\Pi_{n}^c(x_c, \psi^p):=(x_c, \epsilon_n(x_c, v)), \quad\tilde \Pi_n^+(x_c, \psi^p):=(0, \Pi^+(\psi^p-\epsilon_n(x_c, v)),
\\
&\tilde \Pi_n^-(x_c, \psi^p):=(0, \Pi^-(\psi^p-\epsilon_n(x_c, v)),\quad \forall \psi^p\in \Xi^p.
\end{align*}
Then, \eqref{susp-ET-su-0} gives
\begin{align}\label{susp-ET-su}
\begin{split}
&\|\Delta_{m-1}\cdots \Delta_n\tilde\Pi_n^{+}(x_c, \psi_n^p(v))\|
\le C\mu_1^{m-n}\|\tilde\Pi_n^{+}(x_c, \psi_n^p(v))\| \quad {\rm for}~ m\ge n,
\\
&\|\Delta_{m}^{-1}\cdots \Delta_{n-1}^{-1}\tilde\Pi_n^{-}(x_c, \psi_n^p(v))\|
\le C\mu_2^{n-m}\|\tilde\Pi_n^{-}(x_c, \psi_n^p(v))\|
 \quad {\rm for}~ m\le n.
\end{split}
\end{align}
Moreover, we obtain from the last two inequalities of \eqref{new-ET} (see also Remark~\ref{simpl})
and \eqref{center+-} that
\begin{align}\label{susp-ET-c}
\begin{split}
&\|\Delta_{m-1}\cdots\Delta_n\tilde \Pi_n^c(x_c, \psi_n^p(v))\|\le C\nu_+^{m-n}\|\tilde\Pi_n^c(x_c, \psi_n^p(v))\| \quad {\rm for}~ m\ge n,
\\
&\|\Delta_{m}^{-1}\cdots\Delta_{n-1}^{-1}\tilde \Pi_n^c(x_c, \psi_n^p(v))\|\le C\nu_-^{m-n}\|\tilde\Pi_n^c(x_c, \psi_n^p(v))\| \quad {\rm for}~ m\le n.
\end{split}
\end{align}
Combining \eqref{susp-ET-su} with \eqref{susp-ET-c}, we conclude that the linear cocycle generated by $(\Delta_n)_{n\in\mathbb{Z}}$
admits an exponential trichotomy and its dichotomy spectrum satisfies
\eqref{DS-Delt} due to \eqref{def-sig} and Remark \ref{simpl}. Moreover, since $\sup_{n\in \mathbb Z}\|\Delta_n\|<+\infty$ and $\sup_{n\in \mathbb Z}\|\Delta_n^{-1}\|<+\infty$, $(\Delta_n)_{n\in\mathbb{Z}}$ admits a strong exponential trichotomy.
This implies that the system
$(\tilde \Upsilon_n)_{n\in\mathbb{Z}}$ defined by (\ref{linear-susp-1}) has a (nonautonomous) $C^{N-p-1}$
center manifold, which has bounded derivatives up order $N-p-1$ near $(0,0)$, provided that the spectral gap condition
\begin{align}\label{new-gap}
\nu_+^{N-p-1}< \mu_2^{-1},\qquad \nu_-^{N-p-1}>\mu_1
\end{align}
holds (see Proposition~\ref{center}). Thus, since \eqref{NR-cond-1} implies~\eqref{new-gap} (actually it implies that $\nu_+^{N}< \mu_2^{-1}$, $\nu_-^{N}>\mu_1$),
we can conclude  that the equation \eqref{elimi-equa-chan} has a $C^{N-p-1}$ (in $x_c$)
solution, which has bounded derivatives up order $N-p-1$ near $(0,\kappa_n^p(0,v))$, due to the discussion given in the second paragraph of this proof, and therefore we find the desired sequence $(H_n^p(x_c,v))_{n\in\mathbb{Z}}$ given in the formulation of Lemma \ref{rem-cent}. Moreover, it is clear that \eqref{bounded} holds locally since $H_n^p$ is a near-identity map. This completes the proof of this lemma. \hfill $\Box$

Applying Lemma \ref{rem-cent} $N_0$ times, we eliminate all terms $f_n^p(x_c,v)$ for $p=1,\ldots ,N_0$ via $N_0$  conjugacies of class $C^{N-N_0}$ (i.e., $C^{2N_0+1}$) and obtain another sequence of maps, also denoted by $(\hat F_n(x_c,v))_{n\in\mathbb{Z}}$, having the form
\begin{equation*}
    \hat F_n (x_c,v):=\Big(w_n(x_c)
    +o(\|v\|^{N_0}),~A_n^{su}(x_c)v+O(\|v\|^2)\Big),
\end{equation*}
where $w_n(x_c)$ is $C^{N-N_0}$ and $A_n^{su}(x_c)$ is $C^{N-N_0-1}$. Moreover, for a $p\in \{2, \ldots,N_0\}$, assuming that
$$
\partial_v^i(\pi_{su}\hat F_n )(x_c,0)=0\quad{\rm for}~i=2,\ldots ,p-1,
$$
we can rewrite $\hat F_n$ as
\begin{equation*}
    \hat F_n (x_c,v)=\Big(w_n(x_c)
    +o(\|v\|^{N_0}),~A_n^{su}(x_c)v+g_n^p(x_c,v)+o(\|v\|^p)\Big),
\end{equation*}
where $g_n^p(x_c,v):=\partial_v^p (\pi_{su}\hat F_n )(x_c,0)v^p$ is a homogeneous polynomial in variable $v$ of order $p$ with coefficients being $C^{N-N_0-p}$ maps in $x_c$.
\begin{lm}\label{rem-cent-1}
The sequence $(\hat F_n)_{n\in\mathbb{Z}}$ of $C^{N-N_0-p}$ maps is conjugated to a sequence $(\check F_n)_{n\in\mathbb{Z}}$ of $C^{N-N_0-p-1}$ maps given by
$$
\check F_n(x_c,v):=\Big(w_n(x_c)+o(\|v\|^{N_0}),~A_n^{su}(x_c)v+o(\|v\|^p)\Big),
$$
via a sequence $(\hat H_n^p(x_c,v))_{n\in\mathbb{Z}}$, where $\hat H_n^p:X_c\times X_{su}\to X_{su}$ is a homogeneous polynomial in variable $v$ of order $p$ with coefficients being $C^{N-N_0-p-1}$ maps in $x_c$. Moreover,
\begin{align}\label{bounded-2}
\begin{split}
&(\hat H_n^p)^{\pm}(0)=0,\quad d_0(\hat H_n^p)^{\pm}=\Id, \quad \|d_{x}(\hat H_n^p)^{\pm}-\Id\|\le \tilde \delta,
\\
&\|d_{x}^i(\hat H_n^p)^{\pm}\|\le \tilde  M,\quad i=2,\ldots ,N-N_0-p,
\quad {\rm for}~ x=(x_c,v)~{\rm near}~0.
\end{split}
\end{align}
\end{lm}

\noindent {\it Proof of Lemma} \ref{rem-cent-1}.
Let $\hat h_n^p:X_c\times X_{su}\to X_{su}$, $n\in\mathbb Z$ be a sequence of homogeneous polynomials of variable $v$ of order $p$, and consider the transformations
\begin{equation*}
    \hat H_n^p:(x_c,v)\mapsto (x_c,v+\hat h_n^p(x_c,v)).
\end{equation*}
Similarly to the proof of Lemma \ref{rem-cent}, we calculate that
\begin{align*}
&   \hat H^p_{n+1}\circ \hat F_n\circ (\hat H_n^p)^{-1}(x_c,v)
   \\
   & = \bigg(w_n(x_c)+o(\|v\|^{N_0}), ~ A_n^{su}(x_c)v-A_n^{su}(x_c)\hat h_n^p(x_c,v) + g_n^p(x_c,v)+
   \\
   &\qquad \hat h_{n+1}^p(w_n(x_c),A_n^{su}(x_c)v)+o(\|v\|^p) \bigg).
\end{align*}
In order to eliminate the term of order $p$, we solve the following nonautonomous homological equation
\begin{equation}\label{elimi-equa-chan-1}
    - A_n^{su}(x_c)\hat h_n^p(x_c,A_n^{su}(x_c)^{-1}v) + g_n^p(x_c,A_n^{su}(x_c)^{-1}v)+ \hat h_{n+1}^p(w_n(x_c),v)=0
\end{equation}
for $n\in \mathbb Z$. Projecting \eqref{elimi-equa-chan-1} onto $X_i$, $i=1,\ldots ,r$, we get
\begin{equation}\label{elimi-equa-proj}
-A_n^{i}(x_c)\pi_i\hat h_n^p(x_c,A_n^{su}(x_c)^{-1}v) + \pi_ig_n^p(x_c,A_n^{su}(x_c)^{-1}v)+ \pi_i\hat h_{n+1}^p(w_n(x_c),v)=0.
\end{equation}
Comparing \eqref{elimi-equa-proj} with \eqref{elimi-equa-chan}, we understand that if we replace $Dw_n(x_c)$ with $A_n^{i}(x_c)$ and use \eqref{NR-cond-2} instead of \eqref{NR-cond-1}, then, similarly to the proof of Lemma \ref{rem-cent}, we can prove Lemma \ref{rem-cent-1}. \hfill $\Box$

Applying Lemma \ref{rem-cent-1} $N_0-1$ times, we eliminate all terms $g_n^p(x_c,v)$ for $p=2,\ldots ,N_0$ via $N_0-1$ conjugacies of class $C^{N-2N_0}$  (i.e., $C^{N_0+1}$) and obtain another sequence, still denoted by $(\check F_n(x_c,v))_{n\in\mathbb{Z}}$, having the form
\begin{equation*}
    \check F_n (x_c,v):=\Big(w_n(x_c)
    +o(\|v\|^{N_0}),~A_n^{su}(x_c)v+o(\|v\|^{N_0})\Big).
\end{equation*}
Now, we see that $(\check F_n)_{n\in\mathbb{Z}}$ and \eqref{Takens-NF} with $f_n^c(x_c):=w_n(x_c)-A_n^cx_c$
have the same derivatives up to order $N_0$ on their common center manifold $X_c$ and that, by using the smooth cut-off functions, \eqref{zeroo}-\eqref{deriv} hold for $(\check F_n)_{n\in\mathbb{Z}}$ and the sequence given by~\eqref{Takens-NF} (replacing $F_n$ and $G_n$, respectively).

Finally, using Theorem \ref{BS-thm-nonauto}, we can find a sequence $(H_n)_{n\in\mathbb{Z}}$ of $C^k$ maps, which satisfies \eqref{Deriv} and \eqref{1330}, locally conjugating
$(\check F_n)_{n\in\mathbb{Z}}$ to its Takens normal form \eqref{Takens-NF}. Hence, the sequence
$$
(\Psi_n)_{n\in\mathbb{Z}}:=(H_n\circ \hat H_n^{N_0}\circ \cdots\circ \hat H_n^2\circ
H_n^{N_0}\circ \cdots\circ H_n^1\circ \varphi_n)_{n\in\mathbb{Z}}
$$
is the desired $C^k$ conjugacy from the statement of Theorem \ref{Thm-Takens}, where $\varphi_n$ is given by \eqref{trans-center} and $H_n^p$, $\hat H_n^p$ are given by Lemmas \ref{rem-cent} and \ref{rem-cent-1}, respectively. Moreover, $(\Psi_n)_{n\in\mathbb{Z}}$ satisfies \eqref{bounded-psi} due to \eqref{bounded-0}, \eqref{bounded} and  \eqref{bounded-2}. This completes the proof of this theorem. \hfill $\Box$

\section{Appendix: The Belitskii-Samavol Theorem on Hilbert spaces}
\setcounter{equation}{0}


\noindent In what follows, $(X,\|\cdot\|)$ is a Hilbert space. Let $A:X\to X$ be an invertible and  bounded linear operator which admits a strong exponential trichotomy. In other words, we suppose that  there exist projections $\Pi^s,\Pi^c$ and $\Pi^u : X\to X$ as well as constants $\tilde K>0$, $0<\tilde \mu_-\le \tilde \mu_+ < \tilde \lambda_-\le 1\le \tilde \lambda_+< \tilde \varrho_- \le \tilde \varrho_+$ such that:
\begin{itemize}
    \item $\Pi^s+\Pi^c+\Pi^u={\rm Id}$;
    \item $A\Pi^*=\Pi^*A$ for $*=s,c,u$;
    \item for any $n\ge 0$,
    \begin{alignat}{3}
        \|A^n\Pi^s\|&\le \tilde K \tilde \mu_+^{n},\qquad &\|A^{-n}\Pi^s\|&\le \tilde K \tilde \mu_-^{-n},
        \label{contrac-part}
        \\
        \|A^{n}\Pi^c\|&\le  \tilde K \tilde \lambda_+^{n},\qquad &\|A^{-n}\Pi^c\|&\le  \tilde K \tilde \lambda_-^{-n},
        \label{center-part}
        \\
        \|A^n\Pi^u\|&\le  \tilde K \tilde \varrho_+^{n}, \qquad &\|A^{-n}\Pi^u\|&\le  \tilde K \tilde \varrho_-^{-n}.
        \label{expan-part}
    \end{alignat}
\end{itemize}
 Let $X_s:=\Pi^sX$, $X_c:=\Pi^cX$ and $X_u:=\Pi^uX$. Clearly, $X=X_s\oplus X_c\oplus X_u$. Let $N\in \mathbb N$ and   $\delta, M>0$, where we will assume in the sequel that $\delta$ is sufficiently small.
Assume that:
\begin{enumerate}[(H1)]
    \item  \label{hyp-F} $F=A+f : X\to X$ is a $C^N$ diffeomorphism, where $f(0)=0$, $d_0f=0$,  $\|d_xf\|\le \delta$ for all $x\in X$ and $\|d^i_xf\|\le M$ for  all $x\in X$  and  $2\le i \le N$;
     \item \label{hyp-G}  $G=A+g : X\to X$ is a $C^N$ diffeomorphism, where $g(0)=0$, $d_0g=0$,  $\|d_xg\|\le \delta$ for all $x\in X$ and $\|d^i_xg\|\le M$ for all $x\in X$  and  $2\le i \le N$;
     \item \label{spec-cond}  $\tilde \mu_+<\tilde \lambda_-^N$ and  $\tilde \lambda_+^N<\tilde \varrho_-$.
\end{enumerate}

\begin{thm}\label{ck-conjuga}
Take an arbitrary $k\in \mathbb N$. Then, there exists $N_0\in \mathbb N$ with the property that if $A$ admits a strong exponential trichotomy, hypotheses~(H\ref{hyp-F})-(H\ref{spec-cond}) hold with $N\ge N_0$,
and  $F$ and $G$ have the common center manifold $\Gamma^c$ such that $d^i_x(F-G)=0$ for all $1\le i\le N$ and all $x\in  \Gamma^c$, then   $F$ and $G$ are locally $C^k$ conjugated.
\end{thm}

This theorem will be proved using the so-called homotopy method. Our  arguments are inspired by the proof of~\cite[Theorem 3.2]{LL-DCDS16}. In accordance with the hypotheses (H\ref{hyp-F}) and (H\ref{spec-cond}), $F$ has the $C^N$ center-stable manifold
\begin{equation}\label{cs-manifold}
    \Gamma^{sc}= \{ (x_s,x_c)+h^{sc}(x_s,x_c) | (x_s,x_c)\in X_s\oplus X_c\},
\end{equation}
and the $C^N$ center-unstable manifold
\begin{equation}\label{cu-manifold}
    \Gamma^{uc}= \{ (x_u,x_c)+h^{uc}(x_u,x_c) | (x_u,x_c)\in X_u\oplus X_c\},
\end{equation}
where  $h^{sc}: X_s\oplus X_c\to X_u$ and $h^{uc}: X_u\oplus X_c\to X_s$ both are  $C^N$.
The intersection of $\Gamma^{sc}$ and $\Gamma^{uc}$ is the  center manifold of $F$.
We now use the center-stable manifold, center-unstable manifold and center manifold to introduce new axes. Let
\begin{align}\label{new-axes}
    \tilde x_s &:= x_s - h^{uc} (x_u,x_c),\notag\\
    \tilde x_c &:=x_c,\\
    \tilde x_u &:= x_u- h^{sc} (x_s,x_c).\notag
\end{align}
By \cite[Theorem 3.1]{CHT}, we see that Lipschitz constants Lip$(h^{uc})$ and Lip$(h^{sc})$ are both small if Lip$(f) \le \delta$ is small.
Thus, the transformation \eqref{new-axes} is invertible on $X$ with Lipschitz inverse by the  Lipschitz inverse function theorem. Furthermore, it is a  $C^N$ diffeomorphism  on $X$ since $h^{sc}$ and $h^{uc}$ are both $C^N$.

Using the transformation \eqref{new-axes}, we have that  $F$ is  $C^N$ conjugated to a $C^N$ diffeomorphism $\widetilde F$ of the form
\begin{equation*}
    \widetilde F( \tilde x_s, \tilde x_c, \tilde x_u) = (A_s\tilde x_s, A_c\tilde x_c, A_u\tilde x_u)+ \tilde f ( \tilde x_s, \tilde x_c, \tilde x_u),
\end{equation*}
where $A_*:=A|_{X_*}$ for $*=s,c,u$, $\tilde f( \tilde x_s, \tilde x_c, \tilde x_u)= O(\|( \tilde x_s, \tilde x_c, \tilde x_u)\|^2)$, $\Pi^s \tilde f( 0, \tilde x_c, \tilde x_u)=0$ for $(\tilde x_c, \tilde x_u)\in X_c\oplus X_u$, and
$\Pi^u \tilde f ( \tilde x_s, \tilde x_c, 0)=0$ for $( \tilde x_s, \tilde x_c)\in X_s\oplus X_c$.

Similarly, $G$ is also $C^N$ conjugated to a $C^N$ diffeomorphism $\tilde G$ having the form
\begin{equation*}
    \widetilde G( \tilde x_s, \tilde x_c, \tilde x_u) = (A_s\tilde x_s, A_c\tilde x_c, A_u\tilde x_u)+ \tilde g ( \tilde x_s, \tilde x_c, \tilde x_u),
\end{equation*}
where $\tilde g( \tilde x_s, \tilde x_c, \tilde x_u)= O(\|( \tilde x_s, \tilde x_c, \tilde x_u)\|^2)$, $\Pi^s \tilde g ( 0, \tilde x_c, \tilde x_u)=0$ for $(\tilde x_c, \tilde x_u)\in X_c\oplus X_u$, and
$\Pi^u \tilde g ( \tilde x_s, \tilde x_c, 0)=0$ for $( \tilde x_s, \tilde x_c)\in X_s\oplus X_c$.


Let $\widetilde R:=\tilde f-\tilde g$. Since $d^i_x(F-G)=0$ for  $0\le i\le N$ and  $x\in \Gamma^c$, we have that $d^i_{\tilde x_c}\widetilde R=0$ for $0\le i\le N$ and  $\tilde x_c\in X_c$. We split $\widetilde R$ into two parts, one of which is  flat on $X_s\oplus X_c$ while the other is flat on $X_c\oplus X_u$.
\begin{lm}\label{lm-decomp}
The function $\widetilde R$ admits a decomposition, $\widetilde R=\widetilde R_1+\widetilde R_2$, where for $0\le i\le [N/2]$ we have that
\begin{align}
&d^i_{(0,\tilde x_c,\tilde x_u)}\widetilde R_1 =0 \qquad\text{for all}~~(\tilde x_c, \tilde x_u)\in X_c\oplus X_u,
    \label{suc}
    \\
    &d^i_{(\tilde x_s,\tilde x_c,0)}\widetilde R_2 =0 \qquad \text{for all}~~( \tilde x_s, \tilde x_c)\in X_s\oplus X_c.
    \label{ssc}
\end{align}
\end{lm}

\begin{proof}
We
consider the Taylor expansion of $\widetilde R$ at $0$ with respect to $\tilde x_u$,
\begin{equation}\label{splitting}
    \widetilde R(\tilde x)=\sum_{i=0}^{[N/2]} \frac{1}{i!} \frac{\partial^i\widetilde R}{\partial \tilde x_u^i}(\tilde x_s, \tilde x_c, 0)\tilde x_u^i +o(\|\tilde x_u\|^{[N/2]}).
\end{equation}
Take partial derivatives of $\frac{\partial^i\widetilde R}{\partial \tilde x_u^i}(\cdot , \cdot , 0) (\cdot )^i $ at $(0,\tilde x_c,\tilde x_u)$ for all $0\le i\le [N/2]$ and all $(\tilde x_c,\tilde x_u)\in X_c\oplus X_u$.
Observe that
\[ \frac{\partial^{i_1+i_2}}{\partial \tilde x_s^{i_1}\partial \tilde x_c^{i_2}}\left(\frac{\partial^i\widetilde R}{\partial \tilde x_u^i}(0, \tilde x_c, 0)\right) = \frac{\partial^{i_1+i_2+i}\widetilde R}{\partial \tilde x_s^{i_1}\partial \tilde x_c^{i_2}\partial \tilde x_u^i}(0,\tilde x_c,0)=0,\]
for $i_1+i_2\le [N/2]$
since $d^i_{(0,\tilde x_c,0)}\widetilde R=0$ for all $0\le i\le N$ and all ${\tilde x_c}\in X_c$.
Thus, $\widetilde R_1$  chosen to be
\begin{equation*}
 \widetilde R_1(\tilde x):=  \sum_{i=0}^{[N/2]} \frac{1}{i!} \frac{\partial^i\widetilde R}{\partial \tilde x_u^i}(\tilde x_s, \tilde x_c, 0)\tilde x_u^i
\end{equation*}
satisfies \eqref{suc}.
Let
$$\widetilde R_2(\tilde x):=\widetilde R(\tilde x) -\sum_{i=0}^{[N/2]} \frac{1}{i!} \frac{\partial^i\widetilde R}{\partial \tilde x_u^i}(\tilde x_s, \tilde x_c, 0)\tilde x_u^i.$$
We see from \eqref{splitting} that $\widetilde R_2(\tilde x)=o(\|\tilde x_u\|^{[N/2]})$, which implies that \eqref{ssc} holds. The proof is completed.
\end{proof}

For simplicity, we will write $x$ instead of $\tilde x$. Since $X$ is a Hilbert space, we can modify $\tilde f$, $\tilde g$, $\widetilde R_1$ and $\widetilde R_2$ and obtain $C^N$ maps $\hat f, \hat g, \hat R_1$ and $\hat R_2$ defined on $X$ such that
$\hat f(x) =\tilde f(x)$, $\hat g(x) =\tilde g(x)$ and $\hat R_i(x) =\widetilde R_i(x)$ for $i=1,2$ and  $\|x\|\le \frac r 2$. Moreover, $\hat f(x) = \hat g(x) = \hat R_i(x) =0$ for $i=1,2$ and $\|x\|\ge r$.

Let $\hat G_\tau(x) := Ax+\hat  g(x)+\tau \hat R_1(x)$ for $0\le \tau \le 1$ and $x\in X$. We define a map
$
\overline G: X\times [0,1] \to X\times [0,1]
$
by
\[ \overline G (x,\tau) := (\hat G_\tau(x), \tau), \qquad\text{for}~ (x,\tau)\in X\times [0,1]. \]

\begin{lm}\label{lm-homot-method}
Suppose that the following homotopy equation
\begin{equation}\label{homotopy}
    \frac{\partial\hat G_\tau}{\partial x} h - h \circ\overline G = -\hat R_1,
\end{equation}
where $h:X\times [0,1]\to X$ is an unknown map, has a $C^k$ solution with the property that there exists $\hat L>0$ such that
\begin{equation}\label{exist-01}
    \| h(x,\tau)\| \le \hat L \|x\|^2 \qquad \text{for every $\tau \in [0,1] $ and  $\|x\|\le 1$.}
\end{equation}
Then, $\hat G_0$ and $\hat G_1$ are $C^k$ locally conjugated.
\end{lm}

\begin{proof}
By \eqref{homotopy}, we have
\begin{equation*}
    D\overline G \begin{pmatrix} h \\
    1
    \end{pmatrix} = \begin{pmatrix}\frac{\partial \hat G_\tau}{\partial x}& \hat R_1\\
    0&1\end{pmatrix} \begin{pmatrix} h \\
    1
    \end{pmatrix} = \begin{pmatrix} \frac{\partial \hat  G_\tau}{\partial x}h +\hat R_1\\
    1
    \end{pmatrix}=\begin{pmatrix} h \circ \overline G\\
    1
    \end{pmatrix}= \begin{pmatrix} h\\
    1
    \end{pmatrix} \circ \overline G.
\end{equation*}
Set $v:=\begin{pmatrix} h\\
    1
    \end{pmatrix}$ and let  ${\mathcal F}_v^t$ be the flow generated by the vector field $v$. Then, for any $(x_0,\tau_0)\in X\times [0,1]$,  it follows that
\begin{align}
    \frac{d(\overline G ({\mathcal F}_v^t(x_0,\tau_0)) )}{dt} &=     D\overline G ({\mathcal F}_v^t(x_0,\tau_0)) \frac{d({\mathcal F}_v^t(x_0,\tau_0))}{dt}
    \notag\\
    &=   D\overline G ({\mathcal F}_v^t(x_0,\tau_0)) \begin{pmatrix} h({\mathcal F}_v^t(x_0,\tau_0))\\
    1
    \end{pmatrix}\notag
    \\
   &= \begin{pmatrix} h\\
    1
    \end{pmatrix}\circ \overline G ({\mathcal F}_v^t(x_0,\tau_0)),
\end{align}
which means that $\overline G ({\mathcal F}_v^t(x_0,\tau_0))$ is a solution of the ODE that induces the flow ${\mathcal F}_v^t$. This yields that
$$
\overline G ({\mathcal F}_v^t(x_0,\tau_0)) = {\mathcal F}_v^t (\overline G ({\mathcal F}_v^0(x_0,\tau_0)) ) = {\mathcal F}_v^t (\overline G (x_0,\tau_0) ).
$$
In other words, $\overline G$ commutes with ${\mathcal F}_v^t$. Due to \eqref{exist-01}, ${\mathcal F}_v^t(x,0)$ exists on $[0,1]$ for $\|x\|\le \tilde r$ with $\tilde r>0$ being a constant.
Letting ${\mathcal F}^1_v(x,0)=(H(x),1)$ for $\|x\|\le \tilde r$, we derive  that
\begin{align*}
((H\circ \hat   G_0)(x),1) &= {\mathcal F}^1_v (\hat   G_0(x),0) ={\mathcal F}^1_v \circ \overline G(x,0)
\\
&=\overline G \circ {\mathcal F}^1_v (x,0)  = \overline G (H(x),1)= ((\hat   G_1\circ H)(x),1 ).
\end{align*}
Notice that, since $h$ is of class $C^k$,  $H$ is a local $C^k$ diffeomorphism near $0$. This shows that $\hat  G_0$ and $\hat  G_1$ are $C^k$ locally conjugated.
The proof of the lemma is completed.
\end{proof}

\begin{lm}\label{lm-solut-homot}
The equation \eqref{homotopy} has a $C^k$ solution satisfying \eqref{exist-01} for any fixed $k\in\mathbb N$ provided that \eqref{suc} holds for a sufficiently large $N$.
\end{lm}

To prove this lemma, we need several  auxiliary lemmas. The proof of the following lemma is analogous to the proof of~\cite[Lemma 4.4]{LL-CPAM} (in the  particular case when $\Omega$ is taken to be a singleton).
\begin{lm}\label{lm-bound-G} Given any $k\in \mathbb N$,
there exists a constant $L>0$ such that
\begin{equation}\label{derivatives}
    \max_{1\le i\le k+1}\bigg\{ \sup_{(x,\tau)\in X\times [0,1]} \left\|d^i_{(x,\tau)}\hat G^n_\tau\right\|, \sup_{(x,\tau)\in X\times [0,1]} \left\|d^i_{(x,\tau)}\hat  G^{-n}_\tau\right\| \bigg\} \le L^n \qquad \text{for all} \quad n\ge 0,
\end{equation}
where $\hat G^n_\tau$ and $\hat  G^{-n}_\tau$ refer to $n$-th iterates of $\hat  G_\tau$ and $\hat G^{-1}_\tau$, respectively.
\end{lm}

\begin{lm}\label{lm-series}
Provided that \eqref{suc} holds for a sufficiently large $N$, the series
\begin{equation}\label{solution-homo-equa}
    h(x,\tau):= -\sum_{n=1}^{+\infty} \frac{\partial \hat   G_\tau^{-n}}{\partial x} (\hat   G_\tau^{n}(x)) \hat  R_1\circ \hat  G_\tau^{n-1}(x)
\end{equation}
converges to a $C^k$ map satisfying \eqref{exist-01} for any fixed $k\in \mathbb N$.
\end{lm}
\begin{proof}
Since $d_0 \hat R_1 =0$ and $\hat R_1 =0$ for $\|x\|\ge r$,  we see that the derivative of $\hat R_1$ can be made arbitrarily small on  $X$ provided that  $r$ is chosen to be small, which implies that the Lipschitz constant of the nonlinear term of $\hat  G_\tau$ on $X$ is small.
According to \cite[Theorem 3.1 (\romannumeral6)]{CHT},  we have that $\bigcup_{\xi\in X_c\oplus X_u}M_\xi =X$, where
$$M_\xi=\left\{ x\in X : \limsup_{n\to\infty} \frac{1}{n}\log\| \hat G_\tau^n(x) -\hat G_\tau^n(\xi)\| \le \log \gamma \right\} $$
for a certain constant $\tilde \mu_+<\gamma<\tilde \lambda_-$ and $M_\xi\cap M_\zeta=\emptyset$ whenever $\xi\ne \zeta$.
Therefore, for any $x\in X$, there exists a $\xi\in X_c\oplus X_u$ such that $x\in M_\xi$. In addition, by \cite[Lemma 3.3]{CHT}, the sequence of differences $(q_n:=\hat G_\tau^n(x) -\hat G_\tau^n(\xi))_{n\ge 0}$
satisfies  the following equation:
\begin{align}\label{equa-stabel-manifold}
q_n &= A_s^n x_s + \sum_{i=0}^{n-1} A_s^{n-1-i}\Pi^s\{T(\xi_i+q_i) - T(\xi_i) \}
\nonumber\\
&\quad - \sum_{i=n}^{\infty} A_{uc}^{n-1-i}({\rm Id}-\Pi^s)\{T(\xi_i+q_i) - T(\xi_i) \},~~~n\ge 0,
\end{align}
where $A_{uc}$ denotes the restriction of  $diag(A_s,A_c, A_u)$ on $X_c\oplus X_u$, $T:=\hat G_\tau - diag(A_s,A_c, A_u)$ and $\xi_n:=\hat G^n_\tau(\xi)$ for $n\ge 0$.
According to the definition of $M_\xi$, there exists $\epsilon>0$ such that $\gamma+\epsilon < \tilde \lambda_-$ and
\[ C_*:= \sup_{n\ge 0} (\gamma+\epsilon)^{-n} \|q_n\| <\infty. \]
From \eqref{equa-stabel-manifold}, we derive the estimate
\begin{align*}
\|q_n\| &\le \|A_s^n\|\|x_s\|+\sum_{i=0}^{n-1} \|A_s^{n-1-i}\| {\rm Lip}(T)\|q_i\|
+\sum_{i=n}^\infty \|A_{uc}^{n-1-i}\| {\rm Lip}(T) \|q_i\|
\\
&\le \|A_s^n\|\|x_s\| +\left( \sum_{i=0}^{n-1} \|A_s^{n-1-i}\| (\gamma+\epsilon)^i + \sum_{i=n}^\infty \|A_{uc}^{n-1-i}\| (\gamma+\epsilon)^i \right){\rm Lip}(T) C_*,
\end{align*}
which implies by \eqref{contrac-part} and \eqref{center-part} that
$$C_*\le \frac{\tilde K}{1-c}\|x_s\|,$$
where
\begin{align*}
 c&:=\sup_{n\ge 0}\left( \sum_{i=0}^{n-1} \|A_s^{n-1-i}\| (\gamma+\epsilon)^{i-n} + \sum_{i=n}^\infty \|A_{uc}^{n-1-i}\| (\gamma+\epsilon)^{i-n}\right){\rm Lip}(T)
 \\
 & \le \sup_{n\ge 0}\left( \tilde K \sum_{i=0}^{n-1} \left( \frac{\tilde \mu_+}{\gamma+\epsilon}\right)^{n-i} \frac{1}{\tilde \mu_+} + \tilde K \sum_{i=n}^\infty \left( \frac{\gamma+\epsilon}{\tilde \lambda_-}\right)^{i-n} \frac{1}{\tilde \lambda_-} \right) {\rm Lip}(T)
 \\
 &\le \tilde K \left( \frac{1}{\gamma+\epsilon-\tilde \mu_+} + \frac{1}{\tilde \lambda_- -\gamma -\epsilon}  \right){\rm Lip}(T)
 \\
 &<1,
\end{align*}
provided that ${\rm Lip}(T)$ is sufficiently small. Then for all $n\ge 0$
\begin{equation}\label{stab-folia}
\|\Pi^s \hat  G_\tau^{n}(x) \| =    \|\Pi^s \hat G_\tau^{n}(x) - \Pi^s \hat G_\tau^{n}(\xi)  \|\le \tilde K\|q_n\| \le \frac{\tilde K^2}{1-c}(\gamma+\epsilon)^n\|x_s\|.
\end{equation}
On the other hand, by \eqref{suc} and the cut-off argument, we have that for $0\le i\le [N/2]$
\[d_{(0,x_c,x_u)}^i \hat R_1 =0 \qquad \text{for all } \quad (x_c,x_u) \in X_c\oplus X_u. \]
Then it follows from the Taylor formula that
\begin{align*}
    d^i_x \hat R_1 &= d_{(0,x_c,x_u)}^i \hat R_1 + d_{(0,x_c,x_u)}^{i+1} \hat R_1 \cdot x_s + \cdots +
    \\
    &\qquad d_{(0,x_c,x_u)}^{[N/2]} \hat R_1 \cdot (x_s)^{[N/2]-i} + \int_0^1 \frac{(1-t)^{[N/2]-i}}{([N/2]-i)!} d_{(tx_s,x_c,x_u)}^{[N/2]+1} \hat R_1 \cdot (x_s)^{[N/2]+1-i} dt
    \\
    &= \int_0^1 \frac{(1-t)^{[N/2]-i}}{([N/2]-i)!} d_{(tx_s,x_c,x_u)}^{[N/2]+1} \hat R_1 \cdot (x_s)^{[N/2]+1-i} dt,
\end{align*}
which implies that for all $x\in X$
\begin{equation}\label{high-power-stab-varia}
 \|d_x^i\hat R_1\| \le \frac{Q}{([N/2]+1-i)!}\| x_s\|^{[N/2]+1-i}\qquad\text{for all} \quad 0\le i\le k,
\end{equation}
with
\[Q= \sup_{x\in X} \left\|d_x^{[N/2]+1} \hat R_1 \right\| <\infty. \]
Consider $i$, $0\le i\le k$, times differentiation of every term on the right-hand side of \eqref{solution-homo-equa}, which is a sum of several terms. We can see every term in this sum has one factor
$$
(D^j\hat R_1 )\circ \hat  G_\tau^{n-1}(x) \qquad \text{for a certain}\quad 0\le j\le i,
$$
where $D^j \hat R_1$ denotes the function of the $j$-th derivative of $\hat R_1$,
which is dominated by $$\frac{Q}{([N/2]+1-i)!}\left( \frac{\tilde K^2}{1-c} \right)^{[N/2]+1-j}(\gamma+\epsilon)^{(n-1)([N/2]+1-j)}\|x_s \|^{[N/2]+1-j}$$ using \eqref{stab-folia} and \eqref{high-power-stab-varia}. By Lemma \ref{lm-bound-G},
the remaining factors are dominated by $L^n$. Consequently, the series of the $i$-th derivative of \eqref{solution-homo-equa} is absolutely convergent  on any bounded domain provided that $N$ is sufficiently large. This shows that $h$ defined by \eqref{solution-homo-equa} is a $C^k$ map. In particular, \eqref{exist-01} holds with a constant $\hat L>0$ if $N\ge 2$.  The proof is completed.
\end{proof}

\begin{lm}\label{lm-solut}
The map $h$ given by~\eqref{solution-homo-equa} satisfies~\eqref{homotopy}.
\end{lm}
\begin{proof}
We check it directly.
By \eqref{solution-homo-equa}, we see that
\begin{align*}
\bigg(\frac{\partial\hat G_\tau}{\partial x} h - h \circ\overline G\bigg)(x,\tau) &=-\sum_{n=1}^{+\infty} \frac{\partial\hat G_\tau}{\partial x} (x) \frac{\partial \hat  G_\tau^{-n}}{\partial x} (\hat  G_\tau^{n}(x)) \hat R_1 \circ \hat G_\tau^{n-1}(x)
     \\
     & \quad+\sum_{n=1}^{+\infty} \frac{\partial \hat  G_\tau^{-n}}{\partial x} (\hat G_\tau^{n+1}(x)) \hat R_1 \circ \hat G_\tau^{n}(x)
     \\
     &=-\sum_{n=1}^{+\infty}  \frac{\partial \hat G_\tau^{-n+1}}{\partial x} (\hat G_\tau^{n}(x)) \hat R_1 \circ \hat G_\tau^{n-1}(x)
     \\
     & \quad+\sum_{n=1}^{+\infty} \frac{\partial \hat  G_\tau^{-n}}{\partial x} (\hat  G_\tau^{n+1}(x)) \hat R_1\circ \hat  G_\tau^{n}(x)
     \\
     &= - \hat R_1(x),
\end{align*}
which yields~\eqref{homotopy}.
\end{proof}

Now we are ready to prove Lemma \ref{lm-solut-homot}.

\noindent
{\bf Proof of Lemma \ref{lm-solut-homot}.} Clearly, Lemma \ref{lm-series} together with Lemma \ref{lm-solut} implies the desired conclusion. \hfill$\square$

We are now in a position to prove Theorem \ref{ck-conjuga}.

\noindent
{\bf Proof of Theorem \ref{ck-conjuga}.} By Lemmas \ref{lm-decomp}, \ref{lm-homot-method} and \ref{lm-solut-homot}, we have that $\hat G_0$ and $\hat G_1$ are locally $C^k$ conjugated.
Let $\hat F_\tau(x) := Ax + \hat  g(x)+\hat R_1(x) + \tau\hat R_2(x)$ for $x\in X$ and $\tau \in [0, 1]$.
Using the same arguments we used to show that  $\hat G_0$ and $\hat G_1$ are  locally $C^k$ conjugated, one can prove that $\hat F_0$ and $\hat  F_1$ are also locally $C^k$ conjugated. Since $\hat  G_0=\widetilde G$ locally, $\hat  G_1=\hat  F_0$ and $\hat  F_1=\widetilde  F$ locally,
we conclude  that $\widetilde G$ is locally $C^k$ conjugated to $\widetilde F$. Since $F$ is conjugated to $\widetilde F$ and $G$ is  conjugated to $\widetilde G$, it is derived that $F$ and $G$ are locally conjugated. The proof is completed.\hfill$\square$
\section*{Acknowledgements}

The authors are ranked in alphabetic order of their names and should  be treated equally. D. Dragi\v cevi\' c was supported in part by Croatian Science Foundation under the Project IP-2019-04-1239 and by the University of Rijeka under the project uniri-iskusni-prirod-23-98.
X. Tang was supported by NSFC \#12001537, the Start-up Funding of Chongqing Normal University (Grant No. 20XLB033), and the Science and Technology
Research Program of Chongqing Municipal Education Commission (Grant No. KJQN202300540).
W. Zhang was supported in part by NSF-CQ \#CSTB2023NSCQ-JQX0020, NSFC \#12271070, and the Research Project of Chongqing Education Commission CXQT21014.


\begin{thebibliography}{1}


\bibitem{Arnold-book} V. I. Arnold, {\it Geometrical Methods in the Theory of Ordinary Differential Equations.} 2nd ed. Springer, 1988.

\bibitem{AS00}
B.Aulbach and S. Siegmund, A spectral theory for nonautonomous difference equations. {\it Proceedings of the Fifth International Conference on Difference Equations and Applications} 2000, Gordon and Breach.

\bibitem{BDV}
L. Barreira, D. Dragi\v cevi\' c and C. Valls, \emph{Exponential dichotomies with respect to a sequence of norms and admissibility}, Internat. J. Math. \textbf{25} (2014),  1450024, 20 pp.







\bibitem{BK-JDDE02}
G. R. Belitskii and A. Ya Kopanskii, Equivariant Sternberg-Chen theorem. {\it J. Dyn. Differ. Equ.}, {\bf 14} (2002), 349--367.

\bibitem{B-TAMS96}
P. Bonckaert, Partially hyperbolic fixed points with constraints. {\it Trans. Amer. Math. Soc.}, {\bf 348} (1996), 997--1011.





\bibitem{CHT}
X. Y. Chen, J. K. Hale, and B. Tan, Invariant Foliations for $C^1$ Semigroups in Banach Spaces, {\it J.  Differential Equations} {\bf 139}(1997), 283-318.

\bibitem{Chi-Nonlin10}
D. Chillingworth, Dynamics of an impact oscillator near a degenerate graze. {\it Nonlinearity} {\bf 23} (2010), 2723.

\bibitem{CDS-JDDE19} L. V. Cuong, T. S. Doan and S. Siegmund, A Sternberg theorem for nonautonomous differential equations. {\it Journal of Dynamics and Differential Equations}, {\bf 31} (2019), 1279-1299.


\bibitem{Henry-book} D. Henry, {\it Geometric Theory of Semilinear Parabolic Equations} vol. {\bf 840}, Springer, 2006.


\bibitem{CDHM-JSP23}
P. Collet, F. Dunlop, T. Huillet and A. Mardin, A Gibbsian random tree with nearest neighbour interaction. {\it J. Stat. Phys.} {\bf 190} (2023), 78.

\bibitem{DZZ} D. Dragi\v cevi\'c, W. Zhang and W. Zhang,
Smooth linearization of nonautonomous difference equations with a nonuniform dichotomy, {\it Math. Z.}, {\bf 292} (2019), 1175-1193.

\bibitem{DZZ-PLMS20}
D. Dragi\v cevi\'c, W. Zhang and W. Zhang, Smooth linearization of nonautonomous differential equations with a nonuniform dichotomy, {\it Proc. London Math. Soc.} {\bf 121} (2020), 32--50.

\bibitem{DR-CPAA09}
F. Dumortier and C. Rousseau, Study of the cyclicity of some degenerate graphics inside quadratic systems. {\it Commun. Pure Appl. Anal.} {\bf 8} (2009), 1133--1157.

\bibitem{DR-QTDS10}
F. Dumortier and C. Rousseau, Smooth normal linearization of vector fields near lines of singularities. {\it Qualitative Theory of Dynamical Systems}, {\bf 9} (2010), 39--87.

\bibitem{DS-TAMS95}
F. Dumortier and B. Smits, Transition time analysis in singularly perturbed boundary value problems.{\it Trans. Amer. Math. Soc.} {\bf 347} (1995), 4129--4145.

\bibitem{D-Nonlin21}
N. Duignan and H. Dullin, On the $C^{8/3}$-regularisation of simultaneous binary collisions in the planar four-body problem. {\it Nonlinearity} {\bf 34} (2021), 4944.

\bibitem{D-QTDS21}
N. Duignan, Normal forms for manifolds of normally hyperbolic singularities and asymptotic properties of nearby transitions. {\it Qualitative Theory of Dynamical Systems}, {\bf 20} (2021), 26.

\bibitem{Gallay}
Th. Gallay, \emph{A center-stable manifold theorem for differential equations in Banach spaces}, Comm. Math. Phys. \textbf{152} (1993), 249--268.


\bibitem{Palmer}
K. Palmer, A generalization of Hartman's linearization theorem, {\it J. Math. Anal. Appl.} {\bf 41} (1973), 753-758.

\bibitem{Poin}
H. Poincar\'e, Sur le probl\`{e}me des trois corps et les \'equations de la dyanamique, {\it Acta Math.} {\bf 13} (1890), 1-270.

\bibitem{Ray-ETDS09} V. Rayskin, Theorem of Sternberg-Chen modulo the central manifold for Banach spaces. {\it Ergodic Theory and Dynamical Systems}, {\bf 29} (2009), 1965-1978.



\bibitem{IL-AMS98}
Y. Ilyashenko and W. Li, {\it Nonlocal Bifurcations}, volume 66 of Mathematical Surveys and Monographs,
American Mathematical Society, Providence, Rhode Island, 1998.

\bibitem{LL-CPAM}
W. Li and K. Lu, Sternberg theorems for random dynamical systems, {\it Comm. Pure Appl.
Math.}, {\bf 58} (2005), 941-988.

\bibitem{LL-DCDS16} W. Li and K. Lu, Takens theorem for random dynamical systems. {\it Discrete and Continuous Dynamical Systems-B}, {\bf 21} (2016), 3191--3207.

\bibitem{Ma-AnnMath68}
J. Mather, Stability of $C^\infty$-mapping, I–VI, {\it Ann. Math.}, {\bf 87} 89--104.



\bibitem{NPT-IHES83}
S. E. Newhouse, J. Palis and F. Takens, Bifurcations and stability of families of diffeomorphisms. {\it Publications Mathématiques de l'IH\'ES}, {\bf 57} (1983), 5--71.


\bibitem{Rob-Nonlin89}
C. Robinson, Homoclinic bifurcation to a transitive attractor of Lorenz type. {\it Nonlinearity} {\bf 2} (1989), 495--518.

\bibitem{Siegm-JDE02} S. Siegmund, Normal forms for nonautonomous differential equations, {\it J. Differ. Equ.} {\bf 178} (2002) 541-573.

\bibitem{Siegm-JDEA02} S. Siegmund, Block Diagonalization of Linear Difference Equations
{\it J. Differ. Equ. Appl.} {\bf 8} (2002) 177-189.

\bibitem{Sternberg1}
S. Sternberg, Local $C^n$ transformations of the real line, {\em Duke Math. J.}, 24 (1957) 97--102.

\bibitem{Sternberg2}
S. Sternberg, Local contractions and a theorem of Poincar\'{e}, {\em Amer. J. Math.}, 79(1957) 809--824.

\bibitem{Takens-Topol71} F. Takens, Partially hyperbolic fixed points, {\it Topology}, {\bf 10} (1971), 133-147.

\bibitem{WL-JDE08} H. Wu and W. Li, Poincar\'e type theorems for non-autonomous systems. {\it Journal of Differential Equations}, {\bf 245} (2008), 2958-2978.

\end{thebibliography}
\end{document}